%% file: DE2.tex
\documentclass[12pt]{article}
\usepackage[english]{babel}
\usepackage{sfmath}

\usepackage{amssymb}
\usepackage{natbib}

\usepackage[leqno]{mathtools}

\usepackage{amsthm}

\usepackage{graphicx,color}

\usepackage{mathrsfs}
\usepackage{upgreek}
\usepackage{xr-hyper}

\usepackage[bookmarks,pdfnewwindow,plainpages=false]{hyperref}

\hypersetup{
implicit=false,
colorlinks=true,
linkcolor=black,
citecolor=black,
pdfauthor={Victor Ivrii},
pdftitle={Sharp Spectral Asymptotics for Dirac Energy. II. Magnetic Schrodinger operator},
pdfsubject={Sharp Spectral Asymptotics},
pdfkeywords={Microlocal Analysis, Schr\"odinger Operator},
}
\newcommand\MW{{\rm {MW}}}

\definecolor{verylightgray}{gray}{.93}

\usepackage{subfig}
\usepackage{placeins}

\usepackage{enumitem}

\newcommand\W{{\rm W}}

\newcommand\Def{{\overset {\rm {def}}{\ =\ }}}

\newcommand\const{{\rm {const}}}

\newcommand\w{{\rm w}}
\newcommand\boldalpha{{\boldsymbol \alpha }}
\newcommand\boldbeta {{\boldsymbol \beta }}
\newcommand\supp{\operatorname{supp}}

\newcommand\Tr{\operatorname{Tr}}

\newcommand\bR{{\mathbb R}}
\newcommand\bC{{\mathbb C}}

\newcommand\bZ{{\mathbb Z}}
\newcommand\cA{{\mathcal A}}
\newcommand\cB{{\mathcal B}}

\newcommand\cF{{\mathcal F}}
\newcommand\cG{{\mathcal G}}

\newcommand\cI{{\mathcal I}}
\newcommand\cJ{{\mathcal J}}

\newcommand\cR{{\mathcal R}}

\newcommand\cT{{\mathcal T}}

\newtheorem{theorem}{Theorem}[section]

\newtheorem{proposition}[theorem]{Proposition}

\theoremstyle{definition}

\theoremstyle{remark}
\newtheorem{remark}[theorem]{Remark}

\numberwithin{equation}{section}

\newenvironment{claim}[1][{\rm(\theequation)}]{\refstepcounter{equation}%
\begin{trivlist}
\item[{\hskip\labelsep#1}]}{\end{trivlist}\addvspace{10pt}}

\newcounter{note}

\newenvironment{Note}[1]{\refstepcounter{note}{\color{red}$^{\thenote}$}\reversemarginpar
\marginpar{ \hskip-40pt\colorbox{blue!40}{\begin{minipage}{70pt} {\color{red}$^{\thenote}$\ #1}\end{minipage}}}}

\newenvironment{claim*}[1]{\medskip
\begin{trivlist}
\item[{\hskip\labelsep#1}]}{\medskip\end{trivlist}}

\newenvironment{phantomequation}[1][]{\refstepcounter{equation}}{}

\renewcommand\subsubsection{\paragraph{\thesubsubsection}\refstepcounter{subsubsection}}
\begin{document}

%\AddToShipoutPicture*{
%\AtTextCenter{
%\makebox(0,0)[c]{\resizebox{\textwidth}{!}{
%\rotatebox{55}{\textsf{\textbf{\color{lightgray} Preliminary Version}}}}} 
%}
%}

\title{%
Sharp Spectral Asymptotics for Dirac Energy. II. Magnetic Schr\"odinger operator}
\author{Victor Ivrii}
\date{\today}

\maketitle

{\abstract%
I derive sharp semiclassical asymptotics  of $\int |e_h(x,y,0)|^2\omega(x,y)\,dx\,dy$ where $e_h(x,y,\tau)$ is the Schwartz kernel of the spectral projector of Magnetic Schr\"odinger operator and $\omega(x,y)$ is singular as $x=y$. I also consider asymptotics of more general expressions.
\endabstract}

\section{Introduction}\label{sect-0}

This paper is a continuation of \cite{DE1} and I consider
\begin{equation}
I\Def\iint \omega (x,y) e(x,y,\tau)\psi_2(x) e (y,x,\tau) \psi_1(y)\,dx\,dy
\label{0-1}
\end{equation}
where $e(x,y,\tau)$ is the Schwartz kernel of the spectral projector $E(\tau)$ of the (magnetic) Schr\"odinger operator 
\begin{equation}
A= {\frac 1 2}\Bigl(\sum_{j,k}P_jg^{jk}(x)P_k -V\Bigr),\qquad P_j=h D_j-\mu V_j
\label{0-2}
\end{equation}
in $\bR^d$ and $\tau= 0$, $h\to +0$ while $\mu\to +\infty$.

Further, 
\begin{claim}\label{0-3}
$\omega (x,y)\Def \Omega
(x,y; x-y)$ where function $\Omega$ is smooth in 
$B(0,1)\times B(0,1)\times B(\bR^d\setminus 0)$ and homogeneous of degree $-\kappa$ ($0<\kappa<d$) with respect to its third argument \footnote{\label{foot-1} In other words it is Michlin-Calderon-Zygmund kernel.}
\end{claim}
and smooth cut-off functions $\psi_1,\psi_2$.

It follows from \cite{DE1} that $I= \cI +O(h^{1-d-\kappa})$ as $\mu=O(1)$\;
with $\cI$ defined by the same formula but with $e(x,y,\tau)$ replaced by 
\begin{equation}
e^\W_y (x,y,\tau) \Def (2\pi h)^{-d}\int_{g(y,\xi)\le V(y)+2\tau} e^{ih^{-1}\langle x-y,\xi\rangle }\,d\xi.
\label{0-4}
\end{equation}
Then the standard rescaling technique implies the same asymptotics but with the remainder estimate $O(\mu h^{1-d-\kappa})$ provided $1\le \mu=o(h^{-1})$.

Let $d=2$. Then \emph{in the general case\/} it is the best remainder estimate possible while $O(\mu h^{-1})$ is the best possible remainder estimate for 
\begin{equation}
J \Def \int e(x,x,\tau)\psi (x)\,dx;
\label{0-5}
\end{equation}
one needs to consider constant magnetic field 
\begin{equation}
F= g ^{-{\frac 1 2}}(\partial_{x_1}V_2-\partial _{x_2}V_1), \qquad g=\det (g^{jk})^{-1}
\label{0-6}
\end{equation}
and $g^{jk}=\const$, $V=\const$. In this case spectrum consists of \emph{Landau levels}
\begin{equation}
{\frac 1 2} \bigl((2n+1) \mu h f-V\bigr), \qquad n\in \bZ^+, 
\label{0-7}
\end{equation}
of infinite multiplicity and $e(x,x,\tau)=e^\MW(x,\tau)$,
\begin{equation}
e^\MW (x,\tau) = {\frac 1 {2\pi}} \sum_{n\ge 0}
\theta \bigl(2\tau+V-(2n+1)\mu h f\bigr) \mu h ^{-1}f\sqrt{g}
\label{0-8}
\end{equation}
is a \emph{Magnetic Weyl expression}; see \cite{Ivr1}.

On the other hand, \emph{in the generic case\/} the remainder estimate for (\ref{0-5}) is $o(\mu^{-1}h^{-1})$ and the principal part is $\int e^\MW(x,\tau)\psi (x)\,dx$ if $F$ does not vanish and $\mu \le h^{-1}$; otherwise the remainder estimate for \ref{0-5} is $o(\mu^{-1/2}h^{-1})$; moreover in the latter one can consider $h^{-1}\le \mu\le h^{-2}$ and the remainder estimate would be the same but the principal part would be $O(\mu^{-1} h^{-3})$ and the formula could be more complicated (\cite{IRO6, IRO7}).

Now my purpose is to get the sharper remainder estimate for (\ref{0-1}) under the same conditions. This is a very daunting task since for (\ref{0-5}) \emph{periodic trajectories\/} were the main source of trouble and they were broken in the generic case; for (\ref{0-1}) \emph{loops\/} are also the source of trouble, and in the generic case former periodic trajectories generate a lot of loops (see figures \ref{fig-1} and \ref{fig-2}).

As the result, while an asymptotics with the sharp remainder estimate but with the principal part given by very implicit Tauberian formula (\ref{1-4}) is a rather easy corollary of \cite{Ivr1}, Chapter 6 (done in Section \ref{sect-1}), the deriving of an asymptotics with the sharp remainder estimate and rather explicit principal part is much more difficult.

In Section \ref{sect-2} I consider the \emph{weak magnetic field case\/} when $\mu$ is not very large and therefore \emph{magnetic drift\/} is relatively fast and the next winging is distinguishable in the quantum sense from the previous one and replace $T$ in the Tauberian expression by $\epsilon \mu^{-1}$ with certain error estimate; combining with \cite{DE1} I get a remainder estimate (sharp as $\mu \le (h|\log h|)^{-1/6}$ but not very shabby for $\mu \le (h|\log h|)^{-1/3}$) with non-magnetic principal part.

Section \ref{sect-3} is devoted to the \emph{strong magnetic field case\/} $\mu \ge h^{-1/2-\delta}$ when the reasonable remainder estimate could be derived by the method successive approximation and the unperturbed operator is a model operator, admitting explicit calculations.

Finally, in Section \ref{sect-4} I consider the  \emph{strong magnetic field case\/} when both approaches are combined. 

The results of sections \ref{sect-2}--\ref{sect-4} are not always sharp or very very explicit, but could be made either sharp or completely explicit. But some calculations are left to the readers.

\section{Estimates}\label{sect-1}

\subsection{Tauberian Formula}
\label{sect-1-1}

So I am considering operator (\ref{0-2}) 
where $g^{jk}$, $V_j$, $V$ are smooth real-valued functions of $x\in \bR^d$ and
$(g^{jk})$ is positive-definite matrix, $0<h\ll 1$ is a Planck parameter and
$\mu h\le 1$ is a coupling parameter. I assume that $A$ is self-adjoint operator. Then simple rescaling $x\mapsto \mu x$, $h\mapsto \mu h$, $\mu \mapsto 1$ leads us to the remainder estimate $O(\mu h^{1-d-\kappa})$ which is completely sufficient for applications to Multiparticle Quantum Theory.

However I want to improve this remainder estimate under generic non-degeneracy condition. I consider only the most sensitive case $d=2$ rather than $d=3$. 

\subsubsection{}\label{sect-1-1-1} In this paper I assume that
\begin{align}
&V\ge \epsilon ,\label{1-1}\\
&|F|\ge \epsilon,\label{1-2}\\
&|\nabla V/F|\ge \epsilon_0\label{1-3}
\end{align}
where $F$ is an intensity of magnetic field. Then

\begin{proposition}\label{prop-1-1} Under conditions $(\ref{1-1})-(\ref{1-3})$ the contribution of zone $\{|x-y|\ge C\gamma\}$, to the remainder is $O(\mu^{-1}h^{-1}\gamma ^{-\kappa})$ while the main part is given by the same expression $(\ref{0-1})$ with $e(x,y,0)$ replaced by its standard implicit Tauberian approximation with $T\asymp \epsilon \mu $ 
\begin{equation}
e_T (x,y,0)\Def h^{-1}\int _{-\infty}^0 F_{t\to h^{-1}\tau} \bigl({\bar\chi}_T(t) u(x,y,t)\bigr)\,d\tau.
\label{1-4}
\end{equation}
\end{proposition}

\begin{proof} Here and below $U(t)=e^{ih^{-1}tA}$ is the propagator of $A$ and $u(x,y,t)$ is its Schwartz' kernel.

Consider expression (\ref{0-1}) with $\omega (x,y)$ replaced by $\omega_\gamma(x,y)$ which is a cut-off of $\omega(x,y)$ in the zone $\{|x-y|\asymp\gamma\}$ and with the original $\psi_1.\psi_2$ replaced by $1$. Let us replace \emph{one\/} copy of $e(x,y,\tau)$ by $e(x,y, \tau,\tau')= \bigl(e(x,y,\tau)-e(x,y,\tau')\bigr)$ with $\tau' \le \tau$ and the second copy by $e(x,y,\tau'')$ and denote the resulting expression by $I_\gamma (\tau,\tau',\tau'')$.

Now let us use decomposition
\begin{equation}
\omega_\gamma (x,y)=\gamma^{-d-\kappa}\int \psi_{1,\gamma}(x,z)\psi_{2,\gamma}(y,z)\,dz
\label{1-5}
\end{equation}
where $d=2$ now. 

Then $I_\gamma (\tau,\tau',\tau'')$ does not exceed 
\begin{equation}
\sum_j C\gamma^{-\kappa} \|\varphi_j E(\tau,\tau')\varphi _j\|_1 
\label{1-6}
\end{equation}
where $E(\tau,\tau')= E(\tau)-E(\tau')$, $\varphi _j$ are real-valued $\gamma$-admissible functions supported in $C_0\gamma$-vicinities of $z_j$ and $B(z_j,2C_0\gamma)$ are covering of our domain of multiplicity not exceeding $C_0$. Here I used that $\|E(\tau'')\|=1$. Since $E(\tau,\tau')$ is a positive operator and $\varphi_j=\varphi_j^*$,  one can replace trace norm  by the trace itself and get 
\begin{equation}
\sum_j C\gamma^{-\kappa} \Tr \varphi_j E(\tau,\tau' )\varphi _j =
C\gamma^{-\kappa} \Tr E(\tau,\tau') {\bar\psi}
\label{1-7}
\end{equation}
with ${\bar\psi}=\sum_j\varphi_j^2$.

Further, I know from the standard theory \cite{Ivr1} that under conditions (\ref{1-1})-(\ref{1-3})
\begin{equation}
\|  E(\tau, \tau' ) {\bar\psi} \|_1\le Ch^{-2}\bigl(|\tau -\tau'|+ CT^{-1}h\bigr)\qquad \forall \tau ,\tau' \in [-\epsilon,\epsilon], \; T=\epsilon \mu.
\label{1-8}
\end{equation}
and therefore 
\begin{equation}
|  I_\gamma (\tau,\tau',\tau'') | \le C\gamma^{-\kappa} h^{-2}\bigl(|\tau-\tau'|+ CT^{-1}h\bigr)
\label{1-9}
\end{equation}
in the same frames and therefore due to the standard Tauberian arguments I conclude that the contribution of zone $\{|x-y|\asymp\gamma\}$ to the Tauberian remainder estimate does not exceed $C\mu^{-1}h^{-1}\gamma^{-\kappa}$ which implies the statement immediately.\end{proof}

However I need to consider zone $\{|x-y|\le C\gamma \}$, complementary to one above. Assume that
\begin{align}
&\Omega _\kappa (z)= \sum_j D_{z_j} \Omega_{\kappa -1,j} + \Omega_{\kappa -1,0}\label{1-10}\\
\intertext{with index at $\Omega$ showing the degree of the singularity. Then}
&\omega _\kappa (x,x-y) \psi _\gamma(x-y) = \sum_j D_{x_j} \bigl(\omega_{\kappa -1,j} \psi _\gamma \bigr)+ \omega_\kappa \psi'_\gamma+\omega_{\kappa-1}\psi'' _\gamma\label{1-11}
\end{align}
where $\psi_\gamma =\psi ( (x-y)\gamma^{-1})$ with $\psi$ supported in $B(0,1)$ and equal 1 in $B(0, {\frac 1 2})$ while $\psi'_\gamma$ is defined similarly with $\psi'$ supported in $B(0,1)\setminus B(0, {\frac 1 2})$ and the last term gains $1$ in the regularity.

After integration by parts expression $I_{\kappa,\gamma}$, defined by (\ref{0-1}) with $\Omega$ replaced by $\Omega \psi_\gamma$, becomes 
\begin{equation}
-h^{-1} \sum_j\iint \omega_{\kappa-1,j} (x,y) (hD_{x_j}) \bigl( e(x,y,\tau) \cdot e(y,x,\tau)\bigr)\, dxdy
\label{1-12}
\end{equation}
plus two other terms: the term defined by (\ref{0-1}) with kernel $\Omega '_{\kappa,j}$ of the same singularity $\kappa$, without factor $h^{-1}$ and supported in the zone $\{|x-y|\ge \gamma/2\}$ and the term defined by (\ref{0-1}) with kernel $\Omega '_{\kappa-1,j}$ without factor $h^{-1}$ and of singularity $\kappa-1$.

The former could be considered as before yielding to the same remainder estimate 
$O(\mu^{-1}h^{1-d}\gamma^{-\kappa})$. To the latter I can apply the same trick again and again raising power (and these terms are treated in the same manner (but simpler) as I deal below with (\ref{1-12}).

So, one needs to consider (\ref{1-12}) and thus, denoting the second copy of $e(y,x,\tau)$ by $f(y,x,\tau)$ and without using that they are equal
\begin{align}
&(hD_{x_j}) \bigl( e(x,y,\tau) \cdot f(y,x,\tau)\bigr) =\label{1-13}\\
&\bigl( hD_{x_j}e(x,y,\tau)\bigr) f(y,x,\tau)\ -\
e(x,y,\tau) \bigr(f(y,x,\tau)(hD_{x_j})^t\bigr) =\notag\\
&\bigl( P_{j,x}e(x,y,\tau)\bigr) f(y,x,\tau)\ \ -\
e(x,y,\tau)\bigr(f(y,x,\tau)P_{j,x}^t\bigr).\notag
\end{align}
I remind $P_j=hD_j-\mu V_j(x)$ and $P_j^t=-hD_j-\mu V_j(x)$ is the dual operator. I also remind that if $e(x,y,\tau)$ and $f(y,x,\tau)$ are Schwartz kernels of $E(\tau)$ and $F(\tau)$, then $P_{j,x}e(x,y,\tau)$ and $f(y,x,\tau)P_{j,x}^t$ are those of $P_jE(\tau)$ and $F(\tau)P_j$.

Therefore I am interested in the expressions of the type 
\begin{equation}
h^{-1}\iint \omega_{\kappa-1} (x,y) e(x,y,\tau) f(x,y,\tau)
\psi_\gamma \,dxdy.
\label{1-14}
\end{equation}
If $\kappa \le 1$ then replacing $e(x,y,\tau)$ (and $f(y,x,\tau)$) by its standard Tauberian expressions one gets an error not exceeding
$Ch^{-1} \times \mu^{-1}h^{-2}\gamma^{1-\kappa}$ because 
$\|P_jE(\tau)\|\le C_0$, $\|P_jF(\tau)\|\le C_0$ where $F(\tau)$ is an operator with the Schwartz kernel $f(x,y,\tau)$ and also because 
\begin{align}
&\sum_j\| \varphi_j   P_jE(\tau,\tau')\varphi_j \|_1\le 
\sum_j\| \varphi_j   P_jE(\tau,\tau') \|_2 \cdot \|   E(\tau,\tau')\varphi_j  \|_2\le\label{1-15}\\
&\sum _j \| \varphi_j   P_jE(\tau,\tau')\|_2^2 + 
\sum_j \|   E(\tau,\tau') \varphi_j\|_2^2=\notag\\
& \sum_j \Tr \varphi_j   P_jE(\tau,\tau')  P_j ^* \varphi_j +
\sum_j \Tr \varphi_j   E(\tau,\tau') \varphi_j \notag \\
&\le Ch^{-d}\bigl(|\tau -\tau'|+ C\mu^{-1}h\bigr)\qquad \forall \tau,\tau' \in [-\epsilon,\epsilon]\notag
\end{align}
which  also easily follows from \cite{Ivr1}.

So, in this case one gets remainder estimate 
$O\bigl(\mu^{-1}h^{1-d}\gamma^{-\kappa}+\mu^{-1}h^{-d}\gamma^{1-\kappa}\bigr)$ which is optimized to $O(\mu^{-1}h^{1-d-\kappa})$ as $r\asymp h$.

On the other hand, as $1<\kappa <2$ one can apply the same trick again since I did not use the fact that $e(.,.,.)$ and $f(.,.,.)$ coincide; then I arrive to the same estimates with $P_j$ replaced by $P_jP_k$ or even by $P^J\Def P_{j_1}P_{j_2}\cdots P_{j_l}$:
\begin{equation}
\Tr P^J E(\tau,\tau')(P^J)^* \le Ch^{-2}(|\tau -\tau'|+ \mu^{-1}h)\qquad \forall
\tau,\tau'\in [-\epsilon,\epsilon].
\label{1-16}
\end{equation}

Finally, note that as $\kappa \ne 1$, (\ref{1-10}) is always possible. Further, as $\kappa =1$ decomposition (\ref{1-10}) is possible as well provided one adds term $\varkappa (x) |x-y|^{-1}$ with an appropriate coefficient. On the other hand if $\kappa=1$ and $\omega(x,y)=\varkappa (x) |x-y|^{-1}$ then (\ref{1-10}) is also possible but with $\omega_{0,j}(x,y)=\varkappa (x) (x_j-y_j)|x-y|^{-1}\log |x-y|$.

So I arrive to 

\begin{proposition}\label{prop-1-2} Let conditions $(\ref{1-1})-(\ref{1-3})$ be fulfilled. Then 

\medskip
\noindent
{\rm (i)}
As $0<\kappa < 2$ and either $\kappa\ne 1$ or $\kappa=1$ and $\omega$ is replaced by $\omega-\varkappa (x) |x-y|^{-1}$ with an appropriate coefficient $\varkappa(x)$, with the error $O(\mu^{-1}h^{1-d-\kappa})$ one can replace $e(x,y,\tau)$ by its standard Tauberian expression $(\ref{1-4})$ in the formula $(\ref{0-1})$ for $I$.

\medskip
\noindent
{\rm (ii)}
As $\kappa= 1$ and $\omega=\varkappa (x) |x-y|^{-1}$, with the error $O(\mu^{-1}h^{1-d-\kappa}|\log h|)$ one can replace $e(x,y,\tau)$ by its standard Tauberian expression $(\ref{1-4})$ in the formula $(\ref{0-1})$ for $I$.
\end{proposition}

\begin{remark}\label{rem-1-3}
The arguments above show that in an appropriate sense one can consider arbitrary $\kappa \in \bR$ and even in $\bC$.

\medskip
\noindent
(ii) Can one prove the similar result for $I_m$ defined by (\ref{DE1-0-6}), \cite{DE1} with $m\ge 3$?
\end{remark}

\subsubsection{}\label{sect1-1-2}

Since one needs only (\ref{1-16}) rather than (\ref{1-1})-(\ref{1-3}) and (\ref{1-16}) holds as (\ref{1-3}) is replaced by a weaker non-degeneracy condition (see \cite{IRO16})
\begin{claim}
\label{1-17} $V/F$ has only non-degenerate critical points
\end{claim}
I arrive to

\begin{proposition}\label{prop-1-4} Let conditions $(\ref{1-1})$, $(\ref{1-2})$ and $(\ref{1-17})$ be fulfilled. Then $(\ref{1-4})$ and statements (i),(ii) of proposition \ref{prop-1-2} hold.
\end{proposition}

\begin{remark}\label{rem-1-5}
Under certain assumptions (see \cite{IRO4}) this result could be generalized for $d\ge 4$. However in calculations I will need conditions (\ref{1-1})-(\ref{1-3}).\end{remark}

\section{Calculations: Weak Magnetic Field}\label{sect-2}

I am going to assume from now on that $d=2$, $\mu \le h^{-1}$ and conditions (\ref{1-1})-(\ref{1-3}) are fulfilled. In this subsection I assume that the magnetic field is weak enough. I remind that classical particles move along cyclotrons which are circles of the radius $\asymp \mu^{-1}$ (and respectively ellipses if $g^{jk}=\const$) as $g^{jk}=\delta_{jk}$, $V=\const$, $F=\const$ but which drift in more general assumption with the velocity $\mu^{-1}\nabla (V/F)^\perp$. This illustrates the difficulty I am facing: trajectories\footnote{\label{foot-2} Even if one calls them trajectories they are \emph{projections\/} of actual trajectories.} are coming back. In this subsection I consider \emph{weak magnetic field approach\/} when one  gets sharp remainder estimate even if one  ignores returning trajectories.

\subsection{Isotropic Approach}\label{sect-2-1}

\subsubsection{}\label{sect-2-1-1} I want to replace ${\bar\chi}_T(t)$ by ${\bar\chi}_{T'}(t)$ with $T'=\epsilon \mu^{-1}$ in the Tauberian formula and to estimate the corresponding error; in the correct framework this would be equivalent to   using non-magnetic Weyl approximation for $e(x,y,0)$ (see (\ref{DE1-0-3}) \cite{DE1}). 

So let us consider classical Hamiltonian dynamics at Figure \ref{fig-1}.
\begin{figure}[ht]
\centering
\subfloat[hypercycloid]{
\includegraphics{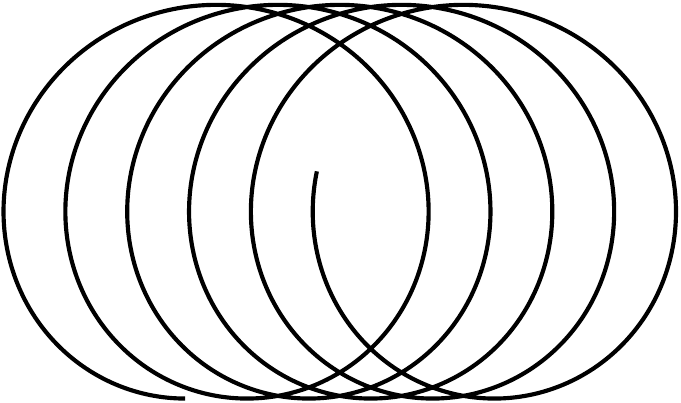}
\label{fig-1a}
}\quad
\subfloat[perturbed hypercycloid]{
\includegraphics{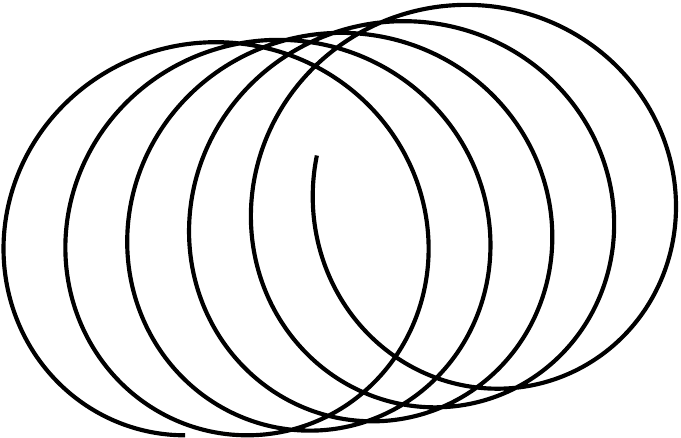}
\label{fig-1b}
}
\caption{\label{fig-1}Classical trajectories}
\end{figure}
Let us notice that while the trajectory is not periodic due to non-degeneracy condition, it is self-intersecting. However, generic point on the trajectory is not a point of the self-intersection, and on each trajectory winging the number of self-intersection points is $\asymp \mu$ and the length of the loop is at least $\asymp\mu^{-1}$ (actually typical loop is of the length $\asymp 1$ but I will make a more precise statement and use it later).

The trajectory as $F=\const$ and $V$ is linear is hypercycloid as on figure \ref{fig-1a}; in more general case the trajectory drifts along some curve and the cyclotron radius changes as on figure \ref{fig-1b} but still one can consider $x$ on the first winging and direct drift at this moment along $x_1$.

Let us parametrize $x$ on the first winging by $\phi \in [0, 2\pi]$; without any loss of the generality one can assume that $\phi \in [0,\pi/2]$; one can reach it by some combination of reflections, reverting time direction and permutation of $x$ and $y$. So, $x$ belongs to the bold arc on figure \ref{fig-2}: 
\begin{figure}[ht]
\centering
\subfloat[$NP$ is the North Pole, $\phi=0$]{
\includegraphics{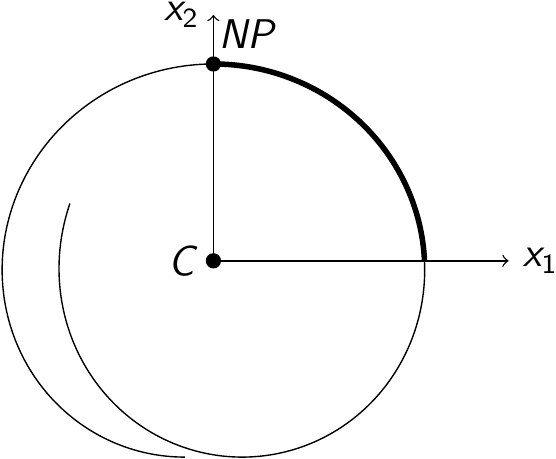}
\label{fig-2a}
}\quad
\subfloat[Loop analysis: right side]{
\includegraphics{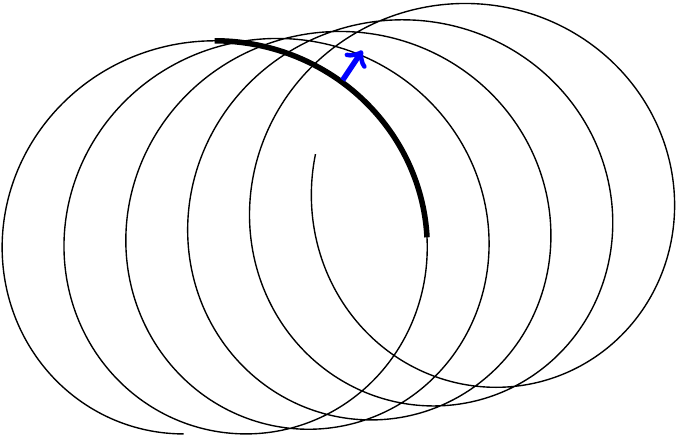}
\label{fig-2b}
}
\caption{\label{fig-2}Classical trajectories: $NP$ is on the top of the first winging}
\end{figure}

So, let us consider some fixed trajectory and point ${\bar x}(\phi)$ on it with $0\le \phi \le \pi/2$; this fixes also ${\bar\xi}(\phi)$\,\footnote{\label{foot-3} In conformal coordinates, where $(g^{jk})$ is proportional to Euclidean metrics, $\phi$ is an angle between $\nabla (V/F)$ and $(p_1,p_2)$.}. 

Note first that in virtue of \cite{Ivr1}

\begin{proposition}\label{prop-2-1} Let $\Phi_t$ be drift flow on the energy level $0$ (see \cite{Ivr1}). Let $\psi_1$ be supported in $c_0\mu^{-1}$-vicinity of ${\bar x}$. Then 
\begin{equation}
|F_{t\to h^{-1}\tau}{\bar\chi}_{\mu (t-t_0)}
\bigl(1- \psi_2(x)\bigr)u(x,y,t)\psi_1(y)|\le Ch^s\qquad \forall \tau: |\tau|\le \epsilon_0\mu^{-1}
\label{2-1}
\end{equation}
as $\psi_1=1$ in $C_0\mu^{-1}$-vicinity of $\Phi_{t_0}({\bar x})$.
\end{proposition}

I claim that 

\begin{proposition}\label{prop-2-2}
Let $Q=Q(x,hD)$ be operator with $(\varepsilon,\mu\varepsilon)$-admissible symbol where 
\begin{equation}
C(\mu^{-1}h|\log h|)^{1/2}\le \varepsilon \le \epsilon_0
\label{2-2}
\end{equation}
Then  for $|t|\le C_0 $, $Q_t=U(-t)QU(t)$ is also $(\varepsilon,\mu\varepsilon)$-admissible operator with $\supp Q_t= \Psi_t (\supp Q)$ with the corresponding Hamiltonian flow $\Psi_t$. 
\end{proposition}

\begin{proof}  (i) Let us rescale $x\mapsto x\mu$, $h\mapsto \hbar=\mu h$, $\mu\mapsto 1$, $\varepsilon \mapsto (\mu h|\log h|)^{1/2}$.

As $F=1$ the proof is really easy since then $e^{2\pi i\hbar^{-1}t_0 A} $ is a standard $\hbar$-FIO corresponding to symplectomorphism $\Psi_{2\pi \mu^{-1}}$ different by $O(\mu^{-1})$ from identical and thus 
\begin{align}
&e^{2\pi i\hbar^{-1} A} =e^{i\mu^{-1}\hbar^{-1}L}\label{2-3}\\
\intertext{where $L$ is $\hbar$-PDO commuting with $A$; then} 
e^{ih^{-1}tA}= e^{i\hbar^{-1}t'L}e^{it''\hbar^{-1}A}\label{2-4}
\end{align}
with $t'= \hbar n$, $n=\lfloor \mu t /(2\pi )+1/2\rfloor$ (and then $|t'|\le C\mu^2 h\ll 1$ and  $t''= \mu t -2\pi n$ (and then $|t''|\le c$) and both of these operators are standard $\hbar$-FIOs.

It is a bit more complicated in the general case but one can always assume that $F({\bar x})=1$ and then (\ref{2-3}) still holds and then the same arguments hold as well.
\end{proof}

\subsubsection{}\label{sect-2-1-2} Let us calculate the distance from $x$ on the upper-right quarter of $0$-th winging to the nearest point $y$ on the right part of $n$-th winging ($n\in \bZ\setminus 0$). One can see easily that
\begin{claim}\label{2-5}
The distance from $x$ to the nearest point $y$ on the right part of $n$-th winging is $r(\phi,n) \asymp \mu^{-2}|n|\sin \phi $ provided $|\sin \phi|\ge C_0 \mu^{-1}|n|$. Otherwise $\ell(\phi,n) \asymp \mu^{-3}n^3$.
\end{claim}
This condition means that the distance $\asymp \mu^{-1} \sin^2\phi $ to the horizontal line passing through the North Pole ($\phi=0$) is larger than the distance one measures; meanwhile the deviation of the enveloping curve from this horizontal line is $O(\mu^{-4}n^2)$ which under this condition is much smaller.
As $n=1$ such distance is shown on figure \ref{fig-2b} above by blue arrow.

Therefore I conclude that point $(x,\xi)$ is distinguishable from each point $\Psi_t (x,\xi)$ with $\epsilon \mu^{-1} \le |t|\le C_0$, residing on the right-halves of the windings provided
\begin{equation}
\epsilon \mu^{-2}\max(\phi ,\mu^{-1})\ge \varepsilon = C(\mu ^{-1}h|\log h|)^{1/2}
\label{2-6}
\end{equation}
where I took the smallest possible $\varepsilon$.

Note that one can satisfy (\ref{2-6}) for any $\phi$ provided 
\begin{equation}
\mu \le \epsilon (h |\log h|)^{-1/5};
\label{2-7}
\end{equation}
let us assume temporarily that this is the case. 

Let us introduce $(\varepsilon,\mu \varepsilon)$-admissible partition
\begin{align}
&\sum _{j\in J} Q_j =I\label{2-8}\\
\intertext{and define}
&Q^\pm = \sum _{j\in J^\pm} Q_j ,\qquad Q^\pm (\rho) = \sum _{j\in J^\pm (\rho)} Q_j
\label{2-9}
\end{align}
where $J^\pm$ refers to elements residing on the right-halves of the trajectories and $J^\pm(\rho)$  refers to those elements of $J^\pm$ which are in $\rho$-vicinity (with respect to $\phi$) of the poles; finally $J({\bar\rho})$ with ${\bar\rho}=\epsilon_0\mu^{-1}$ refers to ``polar caps'' (${\bar\rho}$-vicinities of the poles) and I put them to $J^\pm$ arbitrarily.

So, Tauberian formula (\ref{1-4}) for $e_T$ becomes
\begin{equation}
e_T (x,y,0)= h^{-1}\sum_{(j,k)\in J\times J}
\int _{-\infty}^0 F_{t\to h^{-1}\tau} \bigl({\bar\chi}_T(t) Q_{jx}u(x,y,t)Q_{ky}^t\bigr)\,d\tau 
\label{2-10}
\end{equation}
and therefore
\begin{align}
f_T(x,y)\Def &e_T (x,y,0)-e_{T/2} (x,y,0)=\label{2-11}\\
&h^{-1}\sum_{(j,k)\in \cJ}
\int _{-\infty}^0 F_{t\to h^{-1}\tau} \bigl(\chi_T(t)) Q_{jx}u(x,y,t)Q_{ky}^t\bigr)\,d\tau =\notag\\
&T^{-1} \sum_{(j,k)\in J\times J}
F_{t\to h^{-1}\tau} \bigl({\tilde\chi}_T(t) Q_{jx}u(x,y,t)Q_{ky}^t\bigr)\bigr|_{\tau=0}\notag
\end{align}
with $\chi(t)={\bar\chi}(t)-{\bar\chi}(2t)$,  ${\tilde\chi}(t)=-it^{-1}\chi(t)$; here $\cJ=J\times J$ contains all the pairings but later I also consider $\cJ=J^\pm\times J^\pm$ and $\cJ=J^\pm\times J^\mp$.

Let us consider
\begin{equation}
I'_T= \int \omega(x,y) f_T(x,y)e(y,x,0)\,dxdy.
\label{2-12}
\end{equation}
Due to the same arguments as before
\begin{claim}\label{2-13}
If one  replaces in (\ref{2-12}) $\omega$ by its cut-off  in the zone $\{|x-y|\ge \gamma\}$ then the  result  would not exceed  $CT^{-1}h^{-1}\gamma^{-\kappa}$.
\end{claim}
We apply this estimate with $\gamma \asymp \mu^{-1}T$. Note that as $T\in [C_0,\epsilon \mu] $ zone $\{|x-y|\le \gamma \}$ with $\gamma=\epsilon_0 \mu^{-1}T$ provides a negligible contribution into (\ref{2-12}) due to proposition \ref{prop-2-1} and therefore
\begin{equation}
|I'_T|\le CT^{-1-\kappa}\mu^\kappa h^{-1}\qquad\text{as\ }T\in [C_0,\epsilon \mu].
\label{2-14}
\end{equation}

On the other hand, if I  replace in (\ref{2-12}) $\omega$ by its cut-off   in the zone $\{|x-y|\le \epsilon_1\mu^{-1}T^2\}$ and pick  $\cJ=J^\pm\times J^\pm $ I get a negligible result as well due to proposition \ref{prop-2-2} and (\ref{2-5}).  

So, let us consider $\gamma \in [\epsilon_1 \mu^{-1}T^2, \epsilon \mu^{-1}T]$, $T\in [\epsilon\mu^{-1}, C_0]$ and let us  replace in (\ref{2-12}) $\omega$ by its cut-off in the zone $\{|x-y|\asymp \gamma\}$. Then one can also replace $J^\pm $ by $J^\pm (\rho)$ with $\rho = c\mu T^{-1} \gamma $; the error will be negligible again due to proposition \ref{prop-2-2} and (\ref{2-5}).

To estimate the resulting expression I need an inequality 
\begin{equation}
\|Q(\rho)E(\tau,\tau')Q(\rho)\|_1\le C\rho h^{-2} \bigl(|\tau-\tau'|+ \mu^{-1}h\bigr)
\label{2-15}
\end{equation}
which I will prove a bit  later. Due to (\ref{2-15}) and our standard analysis such modified expression (\ref{2-12}) with $\cJ=J^\pm \times J^\pm$ does not exceed $CT^{-1}\rho \gamma^{-\kappa}h^{-1}$ (with $\gamma = \mu^{-1}T\rho$) i.e.
\begin{equation}
CT^{-1-\kappa}  \rho^{1-\kappa}\mu^\kappa h^{-1}.
\label{2-16}
\end{equation}
Summation with respect to $\rho$ from $\epsilon_0T$ to $1$ results in 
\begin{equation}
CT^{-1-\kappa}\bigl(T^{1-\kappa} + 1 + \updelta_{\kappa1}|\log T|\bigr) \mu^\kappa h^{-1}\asymp C\bigl(T^{-2\kappa} +T^{-1-\kappa}(1 + \updelta_{\kappa1}|\log T|)\bigr)\mu^\kappa h^{-1}
\label{2-17}
\end{equation}
and summation with respect to $T$ from $T_0$ to $T_1$ results in the same expression (\ref{2-17})  with $T$ replaced by $T_0$. In particular, for $T_0=\epsilon \mu^{-1}$ I get 
\begin{equation}
C\bigl(\mu^{3\kappa }+
\mu^{2\kappa+1} +\mu^3 \updelta_{\kappa1}|\log \mu|\bigr) h^{-1}
\label{2-18}
\end{equation}
So far I replaced in just one copy of $e_T(x,y,0)$ in (\ref{0-1}) $T=\epsilon \mu$ by the smaller value (and only in $J^\pm \times J^\pm$ pairs. However exactly the same arguments work for the second copy as well.

So, I arrive to

\begin{proposition}\label{prop-2-3} Let conditions $(\ref{1-1})-(\ref{1-3})$ and $(\ref{2-7})$ be fulfilled. Let us replace in the Tauberian formula for $e_T(x,y,0)$ (plugged into $(\ref{0-1})$) $T=\epsilon_0\mu $ by $T\in [\epsilon \mu^{-1}, \epsilon \mu]$. Then

\medskip
{\rm (i)} As $T_0\ge C_0$ the error does not exceed the right-hand expression of $(\ref{2-14})$;

\medskip
{\rm (ii)}
As $T\le C_0$  the contribution to the error of all pairs $(j,k)\in J^+\times J^+ \cup J^-\times J^-$ does not exceed $(\ref{2-17})$. 

\medskip
{\rm (iii)}
In particular, as $T=\epsilon \mu^{-1}$ this  contribution does not exceed $(\ref{2-18})$.
\end{proposition}

\begin{proof}[Proof of $(\ref{2-15})$] Proof is standard based on the standard calculation of 
\begin{equation}
F_{t\to h^{-1}\tau} \bigl({\bar\chi}_T(t)\Tr \bigl(Q(\rho)U(t)\bigr)\bigr);
\label{2-19}
\end{equation}
I leave details to the reader.
\end{proof}

Actually one can draw conclusions as $(\ref{2-7})$ is violated but our present arguments are too crude anyway.

\subsubsection{}\label{sect-2-1-3} Now let us consider $\cJ=J^\pm \times J^\mp$; due to (\ref{2-13}) one  needs to consider $T\le C_0$ only; so the error in $J^\pm \times J^\pm$ pairs would not exceed (\ref{2-17}) with $T=1$ and one should not bother to get better estimate anyway.

Let analyze the  left halves of trajectories (Figure \ref{fig-3}).
\begin{figure}[ht!]
\centering
\subfloat[Right part]{
\includegraphics{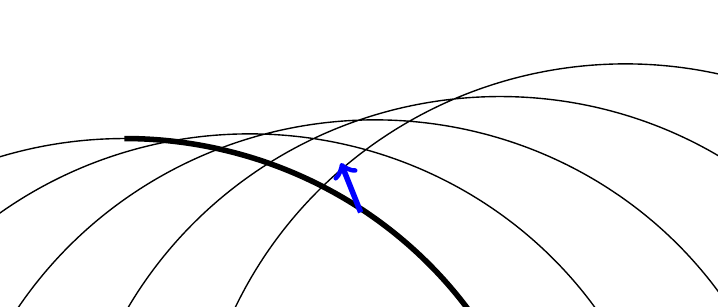}
\label{fig-3a}
}\quad
\subfloat[Left part]{
\includegraphics[scale=2]{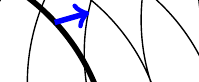}
\label{fig-3b}
}
\caption{\label{fig-3}Classical trajectories: loop analysis: left side}
\end{figure}

In the arguments of the previous subsubsection condition $(j,k)\in J^\pm \times J^\pm$ was used only to dismiss certain distances between $x$ and $y$ as impossible. In particular the same estimates imply that 

\begin{claim}\label{2-20}
If one replaces in (\ref{2-12}) $\omega$ by its cut-off in the zone $\{|x-y|\ge \epsilon \mu^{-2}\}$ the result would not exceed 
$CT^{-1}\mu^{2\kappa}h^{-1}$.
\end{claim}

\subsubsection{}\label{sect-2-1-4} In this subsubsection I consider the case $\epsilon_0\le \phi_n \le \pi/2 -\epsilon_0$, assuming that $\varepsilon \le \mu^{-2}$ i.e.
\begin{equation}
\mu \le \epsilon (h|\log h|)^{-1/3}.
\label{2-21}
\end{equation}

Let $\phi_n$ be the location of $n$-th intersection (or intersection of original quarter with $n$-th winging). Consider point $x$ which is closer to $n$-th intersection than to  any other intersection point. Consider $y$ on $m$-th winging. 

Let us analyze case $n\ne m$ first. Then unless $|\phi_n-\phi_m|\le \epsilon_1$, the distance between $x$ and $y$ is $\asymp \mu^{-1}$ and one should not be concerned since  such points are covered by (\ref{2-20}). On the other hand, if $|\phi_n-\phi_m|\le \epsilon_1$,  the distance between $x$ and $y$ (as $n\ne m$) is at least $\mu^{-2}$ and such pairs are covered by (\ref{2-20}) again.

So, one needs to consider only $m=n$; now one needs to consider contributions of the pairs $(j,k)$ connected by a trajectory, and ``gravitating'' to the same intersection point. Let us joint all the elements ``gravitating'' to $n$-th intersection point and residing on the distance not exceeding $\gamma $ from it 
with $\gamma \in [C\varepsilon,c\mu^{-2}]$; more precisely let us define $Q^\pm_{n,\gamma}$ and $Q_{n,\gamma}$ as the corresponding sums of $Q_j$. Then by the standard methods of \cite{Ivr1} one can prove easily that
\begin{equation}
\|Q_{n,\gamma}E(\tau,\tau')Q_{n,\gamma}\|_1\le C\mu \gamma h^{-2} \bigl(|\tau-\tau'|+ \mu^{-1}h\bigr).
\label{2-22}
\end{equation}
Furthermore, if either $\supp Q_j$ or $\supp Q_k$ was on the distance $\asymp \gamma$ from the intersection point, then the distance between $\supp Q_j$ and $\supp Q_k$ is also $\asymp \gamma$ and therefore the contribution of all such pairs to the error does not exceed $C\mu h^{-1}\gamma^{1-\kappa}$. Then summation with respect to $\gamma$ from $\varepsilon$ to $C\mu^{-2}$ results in 
\begin{equation} 
C\Bigl(\mu^{2\kappa-2} +   \varepsilon^{1-\kappa}+ \updelta_{\kappa1}|\log \mu^2\varepsilon|\Bigr) \mu h^{-1} 
\label{2-23}
\end{equation}
as $T\asymp 1$. Meanwhile 
\begin{claim}\label{2-24}
The contribution to the error of pairs when both elements $Q_j$ and $Q_k$ are supported in $C\varepsilon$-vicinity of the intersection point does not exceed 
$C \mu \varepsilon h^{-1-\kappa}$.
\end{claim}
The proof of (\ref{2-24}) repeats arguments leading to proposition \ref{prop-1-2} with $\varphi_l$ framing $Q_n ^\pm E(\tau,\tau') Q_n^\mp$ from both sides. Clearly $\phi_l$ and $Q_n$ do not commute well (scales are incompatible) but one does not need a commutation here. 

Together (\ref{2-23}) and (\ref{2-24}) imply that the contribution of one ``tick'' (number $n$) to the error does not exceed 
\begin{equation}
C\bigl(\mu^{2\kappa-1}h^{-1}+C\mu ^{1/2}h^{-1/2-\kappa}\bigr).
\label{2-25}
\end{equation}

However there are $\asymp \mu$ ticks and one needs to multiply by $\mu$ resulting in expression
\begin{equation}
C\mu^{2\kappa}h^{-1} + C \mu ^{3/2} h^{-1/2-\kappa}|\log h|^{1/2}.
\label{2-26}
\end{equation}
The first term here is exactly as in (\ref{2-20}) with $T\asymp 1$. Therefore I arrive
\begin{claim}\label{2-27}
If one replaces in the Tauberian expression in pairs $(j,k)\in J^\pm \times J^\mp$ with at least one element, residing in zone $\{\min _{l\in \bZ} |\phi -\pi l/2 |\ge \epsilon_0\}$,   $T= C_0$ by $T=\epsilon$ with arbitrarily small constant $\epsilon$, it would cause the error in (\ref{0-1}) not exceeding (\ref{2-26}). 
\end{claim}

On the other hand, since in the zone in question the distance between connected elements is $\asymp \mu^{-1}$ as $|t|\le \epsilon$ we conclude that replacing $T=\epsilon$ by  some smaller value $T$  would cause the error in (\ref{0-1}) not exceeding $CT^{-1}\mu^\kappa h^{-1}$. Combining this with (\ref{2-27}) we arrive to 

\begin{claim}\label{2-28}
If one replaces in the Tauberian expression in pairs $(j,k)\in J^\pm \times J^\mp$ with at least one element, residing in zone $\{\min _{l\in \bZ} |\phi -\pi l/2 |\ge \epsilon_0\}$,   $T= C_0$ by some smaller value $T$, it would cause an error in (\ref{0-1}) not exceeding 
\begin{equation}
CT^{-1}\mu ^\kappa h^{-1}+ C\mu^{2\kappa}h^{-1} + C \mu ^{3/2} h^{-1/2-\kappa}|\log h|^{1/2};
\label{2-29}
\end{equation}
in particular, as $T\asymp \mu^{-1}$ one gets
\begin{equation}
C\mu ^{\kappa +1} h^{-1}+ C \mu ^{3/2} h^{-1/2-\kappa}|\log h|^{1/2}.
\label{2-30}
\end{equation}
\end{claim}

\subsubsection{}\label{sect-2-1-5} 
Let us consider now $\phi_n\asymp \rho \le \epsilon_0$ (figure \ref{fig-3a}). More precisely, let us consider contribution of pairs $(j,k)\in J^\pm \times J^\mp$ such that both elements $Q_j$ and $Q_k$ are in $\rho$-vicinity of the pole and at least one of them is on the distance $\rho$ from the pole. As I want to use Tauberian formula with $T\le \epsilon_0\rho$ I need three jumps: from $T=C_0$ to $T=C_0\rho$, from $T=C_0\rho$ to $T=\epsilon_0 \rho$ and from $T=\epsilon \rho$ to desired $T$. As $T\asymp \rho$ there would be only two jumps and as $\rho \ge C_0\rho$ there would be only one jump. 

One can see easily that in  the first jump the distance is at least $\epsilon_0\mu^{-1}T^2$ and applying our standard arguments I conclude that the error does not exceed 
\begin{equation}
CT^{-1}\rho (\mu^{-1}T^2)^{-\kappa}h^{-1}\asymp CT^{-1-2\kappa }\rho \mu^\kappa h^{-1};
\label{2-31}
\end{equation}
similarly in the third jump the distance is at least $\epsilon_0\mu^{-1}\rho^2$ and  the error does not exceed 
\begin{equation}
CT^{-1}\rho (\mu^{-1}\rho^2)^{-\kappa}h^{-1}\asymp CT^{-1}\rho^{1-2\kappa} \mu^\kappa h^{-1}.
\label{2-32}
\end{equation}
The second jump as one could see is more tricky.

Note that in zone in question $\phi_n\asymp \mu^{-1}n$ and $|\phi_n-\phi_m|\asymp \mu^{-1}|n-m|$ as $m\ne n$. Since branches intersect under angle $2\phi_n$ I conclude that as $n\ne m$ the distance is measured approximately along the vertical and it is 
$\asymp \mu^{-2}|m-n|\rho $ i.e. at least $\mu^{-2}\rho$. 

Repeating the same arguments as before I conclude that the contribution to the error does not exceed 
\begin{equation}
CT^{-1}\rho (\mu^{-2}\rho )^{-\kappa}h^{-1}\asymp CT^{-\kappa} \mu^{2\kappa } h^{-1}.
\label{2-33}
\end{equation}

As $m=n$ the distance again is measured approximately along the vertical and it is $\asymp \gamma \rho $. Therefore $x$ and $y$ are distinguishable as $\rho\gamma\ge C \varepsilon$ i.e. 
$\gamma \ge \gamma'$ with
\begin{equation}
\gamma' \Def C\mu \varepsilon \rho^{-1}.
\label{2-34}
\end{equation}
Again I need to assume that $\gamma \le \epsilon \mu^{-2}$, i.e. that $\mu^{-2}\ge \gamma'=C\mu \varepsilon \rho^{-1}$ and it is the case even for $\rho \asymp \mu^{-1}$   under assumption (\ref{2-7}). 

So, the contribution of such pairs is estimated by 
$CT^{-1}\mu \gamma (\rho \gamma)^{-\kappa} h^{-1}$ and summation with respect to $\gamma$ from $\varepsilon\rho^{-1}$ to $\mu^{-2}$ results in (\ref{2-23})-like expression 
\begin{equation}
CT^{-1}  \rho ^{-\kappa} \Bigl(\mu^{2\kappa-2}+
\varepsilon ^{1-\kappa}\rho^{\kappa-1}+\updelta_{\kappa1}|\log \mu^2\varepsilon \rho^{-1}|\Bigr)\mu h^{-1}.
\label{2-35}
\end{equation}
Meanwhile the contribution of pairs with both elements in $\gamma'$-vicinity of the intersection point  does not exceed 
\begin{equation}
C(\mu \varepsilon \rho^{-1}) T ^{-1} h^{-1-\kappa} \asymp C\rho^{-2}\mu ^{1/2} h^{-1/2-\kappa}. 
\label{2-36}
\end{equation}
So contribution of one ``tick'' to the error does not exceed $(\ref{2-35})+(\ref{2-36})$ but there are $\asymp \mu \rho$ ticks and taking in account that $\rho\asymp T$ we arrive to (\ref{2-26})-like expression
\begin{equation}
C  \rho ^{-\kappa-1} \mu^{2\kappa-1}  h^{-1} + C\rho^{-1}\mu ^{3/2} h^{-1/2-\kappa}|\log h|^{1/2}. 
\label{2-37}
\end{equation}

Finally, summation with respect to $\rho \gtrsim T$ results in (\ref{2-37}) with $\rho$ replaced by $T$
\begin{align}
&C  T ^{-\kappa-1} \mu^{2\kappa-1}  h^{-1} + CT^{-1}\mu ^{3/2} h^{-1/2-\kappa} |\log h|^{1/2}
\label{2-38}\\
\intertext{while summation of (\ref{2-31}) with $\rho\gtrsim T$  results in }  &CT^{-2\kappa}\mu^\kappa h^{-1}+CT^{-1}\mu^\kappa h^{-1}+ CT^{-1}\mu h^{-1} |\log T|\updelta_{\kappa1/2}.
\label{2-39}
\end{align}
Note that (\ref{2-31}) with $\rho=T$ is just a first term in (\ref{2-39}).

In particular, plugging  into (\ref{2-39}) $T\asymp \mu^{-1}$ and adding to (\ref{2-38}) I get 
\begin{equation}
C    \mu^{3\kappa }  h^{-1} + C \mu ^{5/2} h^{-1/2-\kappa} |\log h|^{1/2}
+C \mu^{\kappa+1} h^{-1}+ C \mu^2 h^{-1} |\log T|\updelta_{\kappa1/2}.
\label{2-40}
\end{equation}
So, I conclude that 
\begin{claim}\label{2-41}
If one replaces in the Tauberian expression in pairs $(j,k)\in J^\pm \times J^\mp$ with at least one element residing in zone $\{\min _{l\in \bZ} |\phi -\pi l  |\le \epsilon_0\}$   $T= C_0$ by some smaller value $T$, it would cause an error in (\ref{0-1}) not exceeding $(\ref{2-38})+(\ref{2-39})$; in particular, as $T\asymp \mu^{-1}$ one get (\ref{2-40}).
\end{claim}

\subsubsection{}\label{sect-2-1-6} Let us consider now zone ${\pi/2}-\epsilon_0\le \phi_n \le \pi/2$ (figure \ref{fig-3b}). Then there exists ${\bar n}$ (the last intersection) such that 
\begin{multline}
\varphi_{\bar n}\le C_0\mu^{-1/2}, \quad 
\varphi_n\asymp \mu^{-1/2} ({\bar n}-n)^{1/2}\ \text{as\ } n\ne {\bar n},
\quad \ell_n\Def |\varphi_n-\varphi_{n+1}|\asymp \mu^{-1}\varphi_n^{-1}, \\
\text{with\ }\varphi_n\Def (\pi/2 -\phi_n); 
\label{2-42}
\end{multline}
so ticks are actually longer than before.

Since windings intersect under angle $2\phi_n$ I conclude that as $n\ne m$ the distance is measured approximately along the horizontal and it is $\asymp \mu^{-2}|m-n|\sin (\varphi_n+\varphi_m)$. Then $x$ and $y$ are distinguishable as long as it is larger than $\varepsilon$. This is always the case as $\mu^{-5/2}\ge \varepsilon$ i.e.  
\begin{equation}
\mu \le \epsilon (h|\log h|)^{-1/4}.
\label{2-43}
\end{equation}

Note that one needs to consider only $n,m$ with $\varphi_n,\varphi_m \in [0, \epsilon_0]$ (since the other pairs are already covered). Then the contribution to an error does not exceed 
\begin{equation*}
\sum _{m<{\bar n}} C \mu^{-1} T_m^{-1}h^{-1} (\mu ^{-2} \varphi_m )^{-\kappa} \asymp
\sum _{m<{\bar n}} C \mu^{5\kappa /2-1}h^{-1} |{\bar n}-m|^{-\kappa/2} \asymp
C \mu^{5\kappa /2-1}h^{-1}\cdot \mu ^{1-\kappa/2} \asymp C\mu^{2\kappa}h^{-1}
\end{equation*}
since $\kappa <2$. This is exactly the first term in (\ref{2-26}). 

Consider now $m=n <{\bar n}$. Then the distance again is measured along the horizontal and it is $\asymp \mu^{-1}|\varphi-\varphi_n| \varphi_n$ ($\varphi =(\pi/2-\phi)$. Then $x$ and $y$ are distinguishable as 
$\mu^{-1}|\varphi-\varphi_n| \varphi_n \ge C \varepsilon$. 
 
So, the contribution of such points is estimated by another (\ref{2-23})-like expression 
\begin{equation}
C \mu \varphi_n^{-\kappa}T_n^{-1}h^{-1} \int \gamma^{-\kappa}\, d\gamma \asymp 
C\mu \varphi_n^{-\kappa}h^{-1} \Bigl(\mu ^{\kappa-1}\ell_n ^{\kappa-1} +
(\varepsilon \varphi_n^{-1})^{1-\kappa}+\updelta_{\kappa1} |\log h|\Bigr)
\label{2-44}
\end{equation}
where integral is taken from $\gamma=\varepsilon \varphi_n^{-1}$ to $\gamma=\mu^{-1}\ell_n$.

Meanwhile the contribution to the error of pairs with both $x$ and $y$ residing in  zone $\{|\varphi-\varphi_n|\le   \varphi_n^{-1}\varepsilon\}$  does not exceed $C\mu  T_n^{-1} h^{-1-\kappa}\varepsilon\varphi_n^{-1}$;  one can prove it easily by the same methods as before. Adding to (\ref{2-44}) I get the contribution of pairs residing near $n$-th tick 
\begin{equation}
C \mu ^{2\kappa}h^{-1}\varphi_n^{-1}+
C\mu^2\varepsilon h^{-1-\kappa}\varphi_n^{-1} \asymp 
C\Bigl(\mu ^{2\kappa -1}h^{-1}+ \mu\varepsilon h^{-1-\kappa}\Bigr)\times \mu^{-1/2}({\bar n}-n)^{-1/2}
\label{2-45}
\end{equation}
where I rewrote the first terms in (\ref{2-44}) due to (\ref{2-42}); other terms of (\ref{2-44}) are dominated by the last term in the left side of (\ref{2-44}).
Note that the first factor on the right is exactly (\ref{2-26}) while the second factor sums to $C$ with respect to $n$.

Finally, let us consider $m=n={\bar n}$ or, more precisely, zone 
$\{|\varphi| \le \epsilon_0\mu^{-1/2}\}$. Then the above arguments still work as long as  $\varphi  \ge C(\mu \varepsilon)^{1/2}= C(\mu h|\log h|)^{1/4}$
and I arrive to the term not exceeding $C\mu^{2\kappa-1/2}h^{-1}$ plus contribution of zone $\{|\varphi| \le C(\mu h|\log h|)^{1/4}\}$, which does not exceed 
\begin{equation}
C(\mu h|\log h|)^{1/4} h^{-1-\kappa} 
\label{2-46}
\end{equation}
which in turn does not exceed $C\mu^{-1}h^{-1-\kappa}$ under condition (\ref{2-7}) and the second term in (\ref{2-26}) otherwise.

Therefore I arrive to

\begin{claim}\label{2-47}
Claims (\ref{2-27}), (\ref{2-28}) (with an extra term $C\mu^{-1}h^{-1-\kappa}$ in estimates) are valid in the ``near equator'' zone as well.
\end{claim}

\subsubsection{}\label{sect-2-1-7} So, when one replaces in (\ref{1-4}) $T=C_0$ by $T=\epsilon \mu^{-1}$, the  error in $\omega$ would not  not exceed $(\ref{2-18})+(\ref{2-40})+C\mu^{-1}h^{-1-\kappa}$ (since $(\ref{2-26})+(\ref{2-30})$ is lesser); removing terms dominated by others I arrive to 
\begin{equation}
C  \bigl(\mu^{2\kappa+1} + C    \mu^{3\kappa }\bigr)  h^{-1} + C \mu ^{5/2} h^{-1/2-\kappa}  + C\mu^{-1}h^{-1-\kappa}.
\label{2-48}
\end{equation}
Thus I arrive to

\begin{proposition}\label{prop-2-4} Let conditions $(\ref{1-1})-(\ref{1-3})$ be fulfilled. Then under assumption $(\ref{2-7})$ $I$ is given by the standard Tauberian formula with $T=\epsilon \mu^{-1}$ with an error not exceeding $(\ref{2-48})$. 
\end{proposition}

This result is very crude: condition (\ref{2-7}) is too restrictive and estimate (\ref{2-48}) is not sharp enough. In the next subsection  improving this approach I get better results.

\subsection{Anisotropic Approach}\label{sect-2-2}

In this subsection I will improve results of subsubsections \ref{sect-2-1-2} and \ref{sect-2-1-5} to weaken condition (\ref{2-7}) and to improve the error estimate; I will  analyze even the case of condition (\ref{2-21}) violated.

\subsubsection{}\label{sect-2-2-1}

The improvement is based on anisotropic partition elements. More precisely  I will take $\varsigma $-scale along trajectories and $\varepsilon $-scale in the perpendicular direction (with respect to $x$, and matching partition with respect to $\xi$) with
\begin{align}
& \varepsilon \le \varsigma \le \epsilon_0\mu^{-1} ,\qquad  \varepsilon \ge C_0\mu\varsigma ^2,\label{2-49}\\
&\varepsilon \varsigma \ge C\mu^{-1}h|\log h|
\label{2-50}
\end{align}
where the second   condition is intended to counter the curvature of trajectories and (\ref{2-48}) will be logarithmic uncertainty principle.

However, one needs to overcome certain obstacles: first, one needs to explain what it actually means (how to quantize such symbol).  Second, scales of different boxes are not compatible and they are not also compatible with the original (\ref{2-8}) partition. Finally,  such operators do not work nicely with FIOs and one needs to overcome this obstacle. 

To quantize a symbol supported in a straight box is easy: it is just a standard Weyl quantization.  In the general case one can transform any rotated box into straight one  by rotation which is a metaplectic operator; then ne can nicely quantize it by Weyl. On the other hand, exactly the same result would be achieved if one quantized the original symbol\footnote{\label{foot-4} Weyl quantization of symbols obtained by linear symplectic transformations in the phase space coincides  with the metaplectic transformation (corresponding to this symplectomorphism) of the quantized symbol, in particular with $(x,\xi)\mapsto (\varrho x, \varrho^{\dag\,-1}\xi)$ which I apply here; $\varrho^\dag$ means a transposed matrix.}

One can see easily that  such operators nicely commute with operators corresponding to $(\varsigma,\mu \varsigma)$-boxes; moreover two such operators corresponding to two rotated boxes with the angle $\theta$ between corresponding axis also nicely commute if 
\begin{equation}
\varepsilon^2 \ge C|\sin \theta| \cdot \mu^{-1}h|\log h|.
\label{2-51}
\end{equation}
In particular, 
\begin{claim}\label{2-52}
If $Q_1,Q_2$ are   operators with the symbol supported in a rotated box $\cB_1, \cB_2$ and if doubled boxes do not intersect, then $Q_1Q_2\equiv 0$.
\end{claim}

So, I need to study propagation of singularities: basically I need to prove that 

\begin{claim}\label{2-53}
Let $|t|\le T\le C_0$. Let $({\bar y},{\bar\eta})=\Psi_t({\bar x},{\bar\xi})$, $Q_1$ be  a quantization of the symbol supported in the rotated $(\varepsilon,\varsigma, \mu \varsigma,\mu \varepsilon)$-vicinity of $({\bar x},{\bar\xi})$  and $Q_2$ be   quantization of the symbol equal to $1$ in $C_0(\varepsilon,\varsigma, \mu \varsigma,\mu \varepsilon)$-vicinity of 
 $({\bar y},{\bar\eta})$. Then $(1-Q_2)U(-t)Q_1\equiv 0$.
\end{claim}

This is not a theorem (yet) because conditions to $\varepsilon$ and $\varsigma$ except (\ref{2-49}) and (\ref{2-50}) are missing. On the figure \ref{fig-4a} the corresponding vicinities are a filled and an empty rectangles  and  here may be much more than one winging.
\begin{figure}[ht]
\centering
\subfloat[Anisotropic propagation]{
\includegraphics{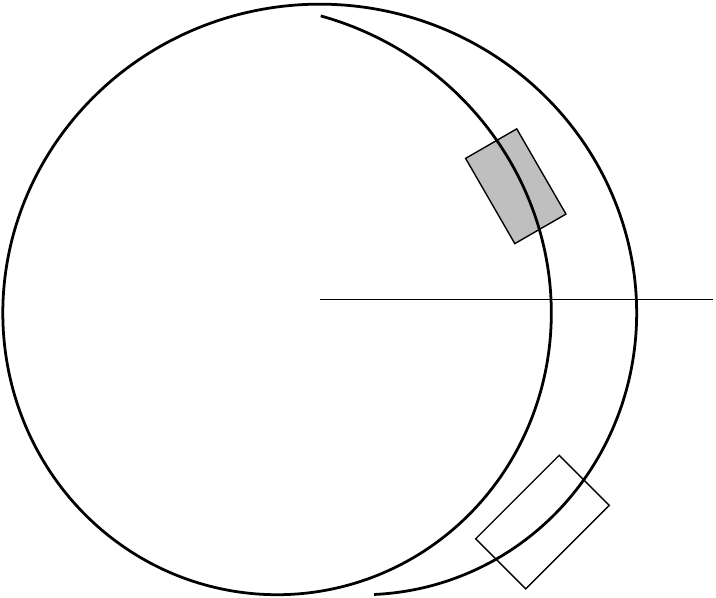}
\label{fig-4a}
}\quad
\subfloat[ to \ref{sect-2-2-2}]{
\includegraphics{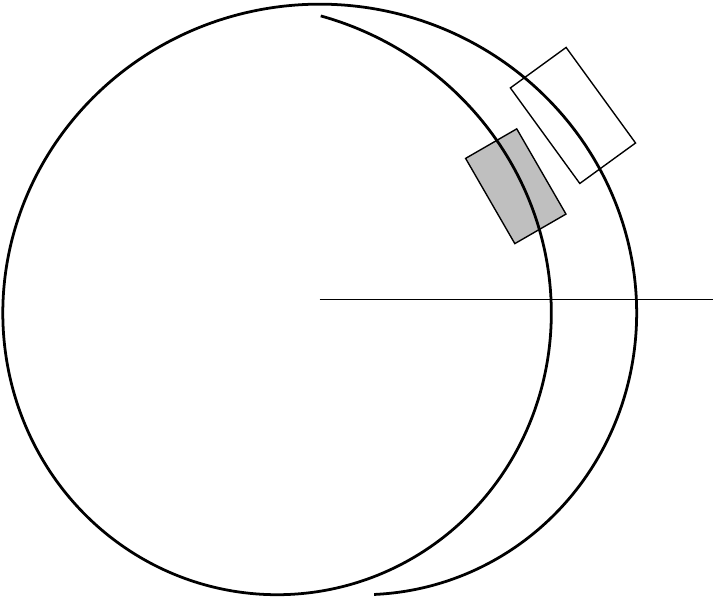}
\label{fig-4b}
}
\caption{\label{fig-4}Anisotropic boxes}
\end{figure}

To prove (\ref{2-53}) I need to elaborate the nature of operator $L$ in (\ref{2-4}), (\ref{2-5}); I remind  that I rescaled $x\mapsto x\mu$, $h\mapsto \hbar=\mu h$, $\mu\mapsto 1$, but did not rescale respectively $t\to t\mu$, however I used $|t'| \le  T$. 

Assume for a moment that $g^{jk}=\const$, and $V_j$, $V$ are linear. In this case $\uppi_x e^{it'H_L}(x,\xi)=\Phi_{ t'} (x)$ where $\Phi_{t'} (x)= x + vt'$ and $\mu^{-1} v$ is the speed of the drift. Then 
$e^{2\pi i\mu^{-1}A}= \cT_{2\pi  v}$ where $\cT_{2\pi  v}$ means a shift
$u(x)\mapsto u(x-2\pi  v)$ (modulo a scalar factor). 

Then obviously the general case 
$e^{i\hbar^{-1}t' L}= \cT_{t' v}\cF_{t'}$ where $\cF_{t'}$ is $\hbar$-FIO with the Hamiltonian map $I+O(\mu^{-1}T^2)$. So, evolution is described by a product of metaplectic operator corresponding to rotation and generated by a quadratic Hamiltonian (i.e. Hamiltonian of the form $q(x, \mu hD)$ where $q$ is quadratic form on $\bR^4$), operator corresponding to the linear phase shift and generated by a linear Hamiltonian (i.e. Hamiltonian of the form $\ell (x, \mu hD)$ where $\ell$ is linear form on $\bR^4$) and the operator corresponding to the symplectomorphism which is almost $I$.

One needs to investigate how such operator acts on the operator which is $\hbar$-quantization o the symbol supported in 
$(\varepsilon' ,\varsigma',\varsigma' ,\varepsilon ')$-box; here I made a replacement $\varsigma\mapsto \varsigma'=\varsigma \mu$, $\varepsilon\mapsto\varepsilon'=\varepsilon \mu$; then (\ref{2-49}),(\ref{2-50}) become 
\begin{equation}
\varepsilon'\le \varsigma '\le \epsilon_0,\qquad \varepsilon '\ge C_0\varsigma^{\prime\,2},\qquad \varepsilon'\varsigma '\ge C\hbar |\log \hbar|.
\label{2-54}
\end{equation}
Now rescaling again $x_1\mapsto x_1/\rho$, $x_2\mapsto x_2$, $D_1\mapsto D_1$, $D_2\mapsto x_2/ \rho$, $\hbar\mapsto \hbar/\rho$ with $\rho= \varepsilon/\varsigma$ I get the last symplectic map $I+O(\varsigma)$ and an isotropic settings. Then I arrive to

\begin{proposition}\label{prop-2-5}
Let $\varepsilon'\ge C(\mu^{-1}\hbar|\log h|)^{1/2}$ i.e.
\begin{equation}
\varepsilon \ge C\mu^{-1}h^{1/2}
\label{2-55}
\end{equation}
and let $Q'_1$ be  a quantization of the symbol supported in the rotated $(\varepsilon',\varsigma',  \varsigma',\varepsilon')$-vicinity of $({\bar x},{\bar\xi})$ and $Q'_2$  be  a quantization of the symbol supported in the rotated $C_0(\varepsilon',\varsigma',  \varsigma',\varepsilon')$-vicinity of 
the same point $({\bar x},{\bar\xi})$.  

Then $(I-Q'_2)\cF_{t'}Q'_1\equiv 0$. 
\end{proposition}

I would not need more analysis of propagation  for the sake of arguments of the next subsubsection (see figure \ref{fig-4b}) since as $|x-y|\ge \varsigma$ one can apply analysis of subsubsection \ref{sect-2-1-2} without any modification.

However, for analysis of $J^\pm \times J^\mp$ pairs in later I will need a bit more analysis since the boxes will be tilted rather than almost parallel.

\subsubsection{}\label{sect-2-2-2}  I replace (\ref{2-8}), (\ref{2-9}) by different partitions while considering the different partitions but $Q^+$ and $Q^-$ still have sense. In this subsubsection  we consider $J^\pm \times J^\pm$ pairs.  Then I should take 
\begin{equation}
\varepsilon =\epsilon_0\mu^{-1}T\rho,\qquad \varsigma = \epsilon_1 \mu^{-1}(T\rho)^{1/2}
\label{2-56}
\end{equation}
where $\rho \asymp T+|\sin \phi|$ as before, the first equality is needed since the distance is $\asymp \mu^{-1}T\rho$ and the second follows from the second inequality in $(\ref{2-49})$\,\footnote{\label{foot-5} Obviously  one should take the largest values possible.};  then due to (\ref{2-50}) one needs to assume that
\begin{equation}
\rho T\ge C(\mu h|\log h|)^{2/3}.
\label{2-57}
\end{equation}
Then (\ref{2-54}) is fulfilled as well.

I remind that $\rho \gtrsim T$ and therefore (\ref{2-57}) holds everywhere as 
\begin{equation}
T\ge C(\mu h|\log h|)^{1/3}.
\label{2-58}
\end{equation}
On the other hand   (\ref{2-57}) holds for $\rho\asymp 1$ as 
\begin{equation}
 T\ge C(\mu h|\log h|)^{2/3}.
\label{2-59}
\end{equation}
In particular, (\ref{2-58}), (\ref{2-59}) are fulfilled with $T=\epsilon \mu^{-1}$ as
\begin{equation}
\mu \le \epsilon (h|\log h|)^{-1/4} 
\label{2-60}
\end{equation}
and
\begin{equation}
\mu \le \epsilon (h|\log h|)^{-2/5} 
\label{2-61}
\end{equation}
respectively.

Since in the calculations of subsubsection \ref{sect-2-1-2} the actual size of of $\varepsilon$ did not matter, only the sheer fact of being disjoint did, I conclude that

\begin{proposition}\label{prop-2-6} Let conditions $(\ref{1-1})-(\ref{1-3})$ be fulfilled. Let us replace in the Tauberian formula for $e_T(x,y,0)$ (plugged into $(\ref{0-1})$) $T=\epsilon_0\mu $ by $T\in [\epsilon \mu^{-1}, C_0]$. Assume that  $(\ref{2-56})$ holds. Then

\medskip
{\rm (i)}
The contribution to the error of all pairs $(j,k)\in J^+\times J^+ \cup J^-\times J^-$ does not exceed $(\ref{2-17})$. 

\medskip
{\rm (ii)}
In particular, as $T=\epsilon \mu^{-1}$ this  contribution does not exceed $(\ref{2-18})$.
\end{proposition}

On the other hand, if only (\ref{2-59}) holds but (\ref{2-58}) fails this analysis remains valid (only) in zone 
\begin{equation}
\rho \ge {\bar\rho}_T \Def CT^{-1}(\mu h|\log h|)^{2/3}
\label{2-62}
\end{equation}
and then I arrive to

\begin{proposition}\label{prop-2-7} Let conditions $(\ref{1-1})-(\ref{1-3})$ be fulfilled. Let us replace in the Tauberian formula for $e_T(x,y,0)$ (plugged into $(\ref{0-1})$) $T=\epsilon_0\mu $ by $T\in [\epsilon \mu^{-1}, C_0]$. Assume that  $(\ref{2-57})$ holds but  $(\ref{2-56})$ fails. Then 

\medskip
{\rm (i)}
The contribution to the error of all pairs $(j,k)\in J^+\times J^+ \cup J^-\times J^-$ with $\rho\ge {\bar\rho}_T$ does not exceed 
\begin{multline}
CT^{-1-\kappa}\bigl({\bar\rho}_T^{1-\kappa} + 1 + \updelta_{\kappa1}|\log T|\bigr) \mu^\kappa h^{-1}\asymp\\
C\bigl(T^{-2} (\mu h|\log h|)^{2(1-\kappa)/3}+T^{-1-\kappa}(1 + \updelta_{\kappa1}|\log T|)\bigr)\mu^\kappa h^{-1}
\label{2-63}
\end{multline}

\medskip
{\rm (ii)}
In particular, as $T=\epsilon \mu^{-1}$ this  contribution does not exceed \begin{equation}
C\mu^{(8+\kappa)/3} h^{-(1+2\kappa)/3}|\log h|^{2(1-\kappa)/3}
+(1 + \updelta_{\kappa1}|\log \mu |) \mu ^{1+2\kappa}  h^{-1}
\label{2-64}
\end{equation}
\end{proposition}

Finally, in this latter case I need to add contribution 
\begin{equation}
C\mu {\bar\rho}_T h^{-1-\kappa}= 
CT^{-1} (\mu ^{5/2}  h|\log h|)^{2/3}h^{-1-\kappa}
\label{2-65}
\end{equation}
of the polar caps $\{\rho \le {\bar\rho}_T$ to the error and I arrive to

\begin{proposition}\label{prop-2-8} In frames of proposition \ref{prop-2-7}

\medskip
{\rm (i)}
The contribution to the error of all pairs $(j,k)\in J^+\times J^+ \cup J^-\times J^-$  does not exceed 
\begin{equation}
CT^{-1} (\mu ^{5/2}  h|\log h|)^{2/3}h^{-1-\kappa}+CT^{-1-\kappa} \mu^\kappa h^{-1}
\label{2-66}
\end{equation}

\medskip
{\rm (ii)}
In particular, as $T=\epsilon \mu^{-1}$ this  contribution does not exceed \begin{equation}
C (\mu ^4  h|\log h|)^{2/3}h^{-1-\kappa}+ C\mu^{2\kappa+1} h^{-1}
\label{2-67}
\end{equation}
\end{proposition}

\subsubsection{}\label{sect-2-2-3} Consider now contribution of $J^\pm \times J^\mp$ pairs. In this case the size of $\varsigma$ does matter from purely geometrical point of view. The other problem is that the boxes introduced in subsubsection \ref{sect-2-2-1} are mutually tilted now. 

To avoid this problem let us partition $Q^+$ and $Q^-$ in different manner. I need to satisfy (\ref{2-49}), (\ref{2-50}). I also need to satisfy $\varepsilon\gtrsim \varsigma\rho$. Therefore I select 
\begin{equation}
\varepsilon = C(\mu^{-1}h|\log h|)^{1/2}\rho^{1/2},\qquad 
\varsigma = C(\mu^{-1}h|\log h|)^{1/2}\rho^{-1/2}.
\label{2-68}
\end{equation}
Here I analyze only case $\rho\le \epsilon_0$ since results of subsubsections \ref{sect-2-1-4} and \ref{sect-2-1-6} could not be improved by considering of anisotropic boxes.

Now 
\begin{claim}\label{2-69} $x$ and $y$ are distinguishable as $\gamma  \ge \varsigma$.
\end{claim}
Really, one can prove easily 
\begin{claim}\label{2-70}
Let conditions (\ref{2-49}), (\ref{2-50}) and $\varepsilon\gtrsim \varsigma\rho$ be fulfilled. Then symbol $b \circ \theta$ with $b(x,\xi)=\beta (\varepsilon^{-1}x_1, \varsigma^{-1}x_2,\varsigma^{-1}\xi_1,  
\varepsilon^{-1}\xi_1)$, regular symbol $\beta$  and $\theta=\theta_\rho$ a rotation of both $x$ and $\xi$ by angle $\rho$ is also quantizable\footnote{\label{foot-6} I remind that I use $\mu^{-1}h$ quantization.} and the corresponding metaplectic transformation of the resulting operator coincides with the quantization of $b$.
\end{claim}

Then encapsulating both rotated and mutually tilted non-intersecting boxes into two $C_0(\varsigma, \varepsilon, \varepsilon,\varsigma)$ boxes (\ref{2-69}).

Note that $\varsigma \lesssim \mu^{-2}$ (which is the length of the tick) as 
\begin{equation}
\rho \ge {\bar\rho}_1 \Def C\mu^3 h|\log h|.
\label{2-71}
\end{equation}
Note that (\ref{2-69}) is satisfied for all $\rho$ (I remind that  $\rho \gtrsim T$ and thus in the final run $\rho \gtrsim \mu^{-1}$ under (\ref{2-60}). On the other hand (\ref{2-69}) holds for $\rho \asymp 1$ under (\ref{2-21}).

Then as in (\ref{2-35}),(\ref{2-36}) the contribution of one tick satisfying (\ref{2-71}) to the error when one replaces $T=C_0\rho$ by $T=\epsilon_0\rho$ does not exceed 
\begin{equation}
CT^{-1}h^{-1}\Bigl(\rho^{-\kappa}\mu^{2\kappa}+ \mu \varsigma h^{-\kappa}\Bigr) \asymp
C\rho^{-1}h^{-1}\Bigl(\rho^{-\kappa}\mu^{2\kappa-1}+ (\mu h|\log h|)^{1/2}\rho^{-1/2} h^{-\kappa}\Bigr)
\label{2-72}
\end{equation}
and contribution of $\asymp \rho \mu$ ticks does not exceed 
\begin{equation}
C \mu h^{-1}\Bigl(\rho^{-\kappa}\mu^{2\kappa-1}+ (\mu h|\log h|)^{1/2}\rho^{-1/2} h^{-\kappa}\Bigr)
\label{2-73}
\end{equation}
and summation with respect to $\rho\gtrsim T$ returns the same expression  with $\rho$ replaced by $T$
\begin{equation}
C \mu h^{-1}\Bigl(T^{-\kappa}\mu^{2\kappa-1}+ (\mu h|\log h|)^{1/2}T^{-1/2} h^{-\kappa}\Bigr).
\label{2-74}
\end{equation}
Noting that contribution of $\{t:|t|\ge T, t \not\asymp \rho\}$ (as it was derived in subsubsection \ref{sect-2-1-5}) does not exceed  
\begin{equation}
C\mu^\kappa T^{-1} h^{-1}\bigl(1+T^{1-2\kappa}+ \updelta_{\kappa\frac12}|\log T|\bigr)
\label{2-75}
\end{equation}
I arrive to

\begin{proposition}\label{prop-2-9} 
Let us replace in the Tauberian expression in pairs $(j,k)\in J^\pm \times J^\mp$ with at least one element residing in zone $\{\min _{l\in \bZ} |\phi -\pi l  |\le \epsilon_0\}$  value $T= C_0$ by some smaller value $T$ satisfying 
\begin{equation}
T\ge \max (\epsilon \mu^{-1},C\mu^3h|\log h|).
\label{2-76}
\end{equation}
Then this would cause an error in $(\ref{0-1})$ not exceeding  $(\ref{2-74})+(\ref{2-75})$.

In particular, as $T\asymp \mu^{-1}$ and $(\ref{2-60})$ is fulfilled the error does not exceed 
\begin{equation}
C   \mu^{3\kappa}h^{-1}  + C\mu^2 h^{-1/2-\kappa}|\log h|^{1/2}+
C\mu^{\kappa +1} h^{-1} +C\mu^{-1}h^{-1-\kappa}.
\label{2-77}
\end{equation}
\end{proposition}

Here (\ref{2-75}) with $T=\mu^{-1}$ brought only one term which is not necessarily dominated by (\ref{2-74}) with $T=\mu^{-1}$ or $C\mu^{-1}h^{-1-\kappa}$. Therefore in comparison with   (\ref{2-40}) I weakened condition (\ref{2-7}) to (\ref{2-60}) and gained factor $\mu^{-1/2}$ in the second term. 

Assume now that 
\begin{equation}
 \epsilon \mu^{-1}\le T\le {\bar\rho}_1\Def C\min(C\mu^3h|\log h|,1).
\label{2-78}
\end{equation}
Again let us replace first $T=C_0\rho$ by $T=\epsilon_0\rho$. 
Then   the contribution of zone $\{|\sin \phi|\asymp \rho\}$ with  $\epsilon_0T\le \rho \le {\bar\rho}_1$ does not exceed 
$C\rho^{-1}\times \rho h^{-1-\kappa}\asymp h^{-1-\kappa}$; I remind that the polar caps I already covered in subsubsections \ref{sect-2-2-1}, \ref{sect-2-2-2} and excepted them from analysis here. Then summation with respect to $\rho$ results in 
\begin{equation}
Ch^{-1-\kappa}|\log T/{\bar\rho}_1|.
\label{2-79}
\end{equation}
Meanwhile, as $\mu \le (h|\log h|)^{-1/3}$ contribution of zone $\{|\sin \phi|\ge   {\bar\rho}_1\}$ does not exceed (\ref{2-74}) with $T\asymp {\bar\rho}_1$ which one can see easily that it is less than (\ref{2-79}). On the other hand, as $\mu \ge (h|\log h|)^{-1/3}$ one should take ${\bar\rho}_1=1$ and this latter zone disappears. 

Further, as $\mu^{-1}\rho |t|\le h|\log h|$ one should reconsider contribution of zone $\{|t|\not\asymp\rho, \rho \le C(\mu h|\log h|)^{1/2}\}$. This however falls inside of the ``polar cap''.
So, I arrive immediately to

\begin{proposition}\label{prop-2-10}
Let us replace in the Tauberian expression in pairs $(j,k)\in J^\pm \times J^\mp$ with at least one element residing in zone $\{\min _{l\in \bZ} |\phi -\pi l  |\le \epsilon_0\}$  (with the exception of polar caps) value $T= C_0$ by some smaller value $T$ satisfying $(\ref{2-78})$. Then this would cause an error in $(\ref{0-1})$ not exceeding  $(\ref{2-78})+(\ref{2-75}$.

In particular, as $T\asymp \mu^{-1}$ and $(\ref{2-60})$ fails  this error does not exceed 
\begin{equation}
 Ch^{-1-\kappa}|\log (\mu^4h|\log h|)| + C\mu^{\kappa +1} h^{-1} .
\label{2-80}
\end{equation}
\end{proposition}

\subsubsection{}\label{sect-2-2-4} Now I want to combine results of two previous subsubsections as $T=\epsilon \mu^{-1}$. I remind  that each of them contains two statements, as $T$ exceeds some critical value or below it; these critical values are $T_1^*=C(\mu h|\log h|)^{1/3}$ and $T_2^*=C\mu^3 h|\log h|$. Let us compare them and also with $\epsilon \mu^{-1}$. 

\begin{proposition}\label{prop-2-11} Let conditions $(\ref{0-3})$ and $(\ref{1-1})-(\ref{1-3})$ be fulfilled. Let us replace in the Tauberian formula $T=C_0$ by $T=\epsilon \mu^{-1}$.  

\medskip
\noindent
{\rm (i)} Assume that condition $(\ref{2-60})$ holds. Then the error does not exceed $(\ref{2-18})+(\ref{2-74})+(\ref{2-75})$; in particular, as $T\asymp \mu^{-1}$ an error does not exceed \begin{equation}
C\mu^{2\kappa+1}h^{-1} + C\mu^{3\kappa}h^{-1} +
C\mu^2 h^{-1/2-\kappa}|\log h|^{1/2} +C\mu^{-1}h^{-1-\kappa}.
\label{2-81}
\end{equation}

\medskip
\noindent
{\rm (ii)} Assume that  but $(\ref{2-60})$ fails but $(\ref{2-61})$ holds. Then  
\begin{enumerate}[label=\rm (\alph*)]
\item as $T\ge T^*_2=C\mu^3h|\log h|$ an  error does not exceed $(\ref{2-18})+(\ref{2-74})+(\ref{2-75})$, 
\item as $T^*_2\ge T\ge T^*_1=C(\mu h|\log h|)^{1/3}$ an error does not exceed $(\ref{2-18})+(\ref{2-79})+(\ref{2-75})$,
\item as $T_1^*\ge T\epsilon \mu^{-1}$  an  error does not exceed $(\ref{2-66})+(\ref{2-79})+(\ref{2-75})$.
In particular, as $T\asymp \mu^{-1}$ an error does not exceed $(\ref{2-67})$.
\end{enumerate}
\end{proposition}

\begin{proof} (i)  Note that   $T_1^*\ge T_2^*$  iff (\ref{2-60}) holds. However in this case $T_2^*\le \epsilon \mu^{-1}$ which means that the error does not exceed $(\ref{2-18})+(\ref{2-74})+(\ref{2-75})$. As $T=\epsilon \mu^{-1}$  reducing dominated terms I arrive to (\ref{2-81}).

\medskip
\noindent
(ii) Assume that (\ref{2-60}) fails but (\ref{2-61}) holds. Then $T^*\le \epsilon\mu^{-1} \le T_1^* \le T_2^* $ and
\begin{enumerate}[label=\rm (\alph*)]
\item As (\ref{2-21}) holds and $T\in [T_2^*,C_0]$ an error does not exceed $(\ref{2-18})+(\ref{2-74})+(\ref{2-75})$;
\item As $T\in  [T^*_1,T_2^*]$ an error does not exceed $(\ref{2-18})+(\ref{2-79})+(\ref{2-75})$;
\item As $T\in [\epsilon \mu^{-1},T_1^*]$  
$(\ref{2-66})+(\ref{2-79})+(\ref{2-75})$. As $T=\epsilon \mu^{-1}$  reducing dominated terms I arrive to  $(\ref{2-67})$.
\end{enumerate}
\end{proof}

\subsection{Main Theorem}
\label{sect-2-3}
Now I can prove the main theorem of this section: 

\begin{theorem}\label{thm-2-12} Let conditions $(\ref{0-3})$ and $(\ref{1-1})-(\ref{1-3})$ be fulfilled. Then $I$ is given by the standard non-magnetic Weyl expression\footnote{\label{foot-7} See (\ref{DE1-2-27})-(\ref{DE1-2-29}), (\ref{DE1-2-7}) \cite{DE1}.}
with the remainder estimate given by $(\ref{2-81})$ if $(\ref{2-60})$ holds and by $(\ref{2-67})$ if   $(\ref{2-60})$ fails but $(\ref{2-61})$ holds.
\end{theorem}

\begin{proof}
In view of proposition \ref{prop-2-11} one should consider $\cI$ defined by a standard Tauberian expression with $T=\epsilon \mu^{-1}$. Scaling $x\mapsto \mu x$, $\mu \mapsto 1$ and $h\mapsto \mu h$ I arrive in view of the proof of Proposition \ref{DE1-prop-2-6} \cite{DE1} to expression
\begin{equation}
\cI \sim \sum_{m,n} \varkappa_{mn} h^{-2-\kappa + m+n} \mu ^m;
\label{2-82}
\end{equation}
however since for real kernel $\omega$ $I$ must be real as well, and complex-conjugation is equivalent to $\mu\mapsto -\mu$, only even powers of $\mu$ are allowed; then modulo $O(\mu^2h^{-\kappa})$ I arrive to the same expression with $\mu=0$. Note that $\mu^2h^{-\kappa}$ is well below the announced remainder estimate.

Due to \cite{DE1} again this expression equals to the standard Weyl expression modulo $O(h^{-\kappa})$.
\end{proof}

\section{Calculations: Strong Magnetic Field }
\label{sect-3} 

\subsection{Calculations for a Model Operator}
\label{sect-3-1}
\nopagebreak
\subsubsection{}\label{sect-3-1-1} Consider first a model operator
\begin{equation}
A= {\frac 1 2}\Bigl((hD_2)^2 +(hD_1-\mu x_2)^2 - W- 2v x_2\Bigr),\qquad W=\const>0,\ v=\const.
\label{3-1}
\end{equation}
 
Then
\begin{equation}
A= {\frac 1 2}\Bigl((hD_2)^2 +(hD_1-\mu x_2+\mu^{-1}v)^2 \mu h- (W+ \mu^{-2}v^2-\mu h)- 2v \mu^{-1}hD_1\Bigr)
\label{3-2}
\end{equation}
and therefore with (\ref{2-3}) holds with $L=\mu B$,
\begin{equation}
B= \pi \bigl( (W+ \mu^{-2}v^2-\mu h)- 2v \mu^{-1}hD_1\bigr).
\label{3-3}
\end{equation}

Let us consider the corresponding spectral problem with $W=0$ (so including a constant part ${\frac 1 2}W$ of the potential into a spectral parameter):
\begin{equation}
{\frac 1 2}\Bigl((hD_2)^2 +(hD_1-\mu x_2)^2 - 2v x_2\Bigr)e=\lambda e
\label{3-4}
\end{equation}
with $e=e(x,y,\lambda)$.

Making a (partial) unitary\footnote{\label{foot-8} So, with the factor $(2\pi)^{-1/2}$.} Fourier transform with respect to $x_1$ one gets
\begin{equation}
{\frac 1 2}\Bigl((hD_2)^2 +(h\xi_1-\mu x_2)^2 - 2v x_2\Bigr){\hat e}=\lambda {\hat e}
\label{3-5}
\end{equation}
with ${\hat e}={\hat e}(\xi_1,x_2,\eta_1,y_2)$ or equivalently
\begin{equation}
\Bigl((hD_2)^2 +(h\xi_1-\mu x_2+v\mu^{-1})^2 \Bigr){\hat e}=(2\lambda + 2v\mu^{-1}h\xi_1 + v^2\mu^{-2}) {\hat e};
\label{3-6}
\end{equation}
therefore
\begin{spreadlines}{10pt}
\begin{multline}
{\hat e}= \mu^{1/2}h^{-1/2}\sum_{n: (2n+1)\mu h \le 2\lambda +2v \mu^{-1}h\xi_1 +v^2\mu^{-2}} 
\Upsilon_n \bigl( \mu^{1/2}h^{-1/2}(-\mu^{-1}h\xi_1+x_2-v\mu^{-2})\bigr)\times\\
\Upsilon_n \bigl( \mu^{1/2}h^{-1/2}(-\mu^{-1}h\xi_1+y_2-v\mu^{-2})\bigr)
e^{-i \xi_1 y_1}\delta (\xi_1-\eta_1)
\label{3-7}
\end{multline}
\end{spreadlines}
where $\Upsilon_n$ are (real orthonormal) Hermite functions.  Therefore excluding again ${\frac 1 2}W$ from the spectral parameter one gets
\begin{spreadlines}{10pt}
\begin{multline}
e(x,y,0)=(2\pi)^{-1} \mu^{1/2}h^{-1/2}\sum_{n}
\int_{(2n+1)\mu h \le W +2v \mu^{-1}h\xi_1 +v^2\mu^{-2}} \times\\
\Upsilon_n \bigl( \mu^{1/2}h^{-1/2}(-\mu^{-1}h\xi_1+x_2-v\mu^{-2})\bigr)
\Upsilon_n \bigl( \mu^{1/2}h^{-1/2}(-\mu^{-1}h\xi_1+y_2-v\mu^{-2})\bigr)
e^{i \xi_1 (x_1-y_1)}\,d\xi_1.
\label{3-8}
\end{multline}
\end{spreadlines}
\vskip2pt\noindent
Plugging $\xi_1=\mu h^{-1}\bigl({\frac 1 2}(x_2+y_2)+\zeta -v\mu^{-2}\bigr)$ 
one can rewrite (\ref{3-8}) as
\begin{spreadlines}{8pt}
\begin{multline}
e(x,y,0)=(2\pi)^{-1} \mu^{3/2}h^{-3/2}
e^{i\mu h^{-1} ({\frac 1 2}(x_2+y_2)-v\mu^{-2})(x_1-y_1)} \times\\
\sum_{n}
\int_{(2n+1)\mu h \le W +v (x_2+y_2) +2v \zeta -v^2\mu^{-2}} \times\\
\Upsilon_n \Bigl( \mu^{1/2}h^{-1/2}({\frac 1 2}(x_2-y_2)-\zeta )\Bigr)
\Upsilon_n \Bigl( \mu^{1/2}h^{-1/2}({\frac 1 2}(x_2-y_2)+\zeta)\Bigr)
e^{-i\mu h^{-1}\zeta (x_1-y_1)}\,d\zeta
\label{3-9}
\end{multline}
\end{spreadlines}
where factor $e^{i\mu h^{-1}({\frac 1 2}(x_2+y_2)-v\mu^{-2})(x_1-y_1)}$ cancels with the adjoint factor coming from $e(y,x,0)$ in the final calculation of  $I=I_2$. Note that $W+v(x_2+y_2)$ is potential $V$ calculated at ${\frac 1 2}(x_2+y_2)$. 

In particular as $v=0$ (degenerate case) one gets
\begin{spreadlines}{10pt}
\begin{multline}
e(x,y,0)=e_W^\MW (x,y,0)\Def (2\pi)^{-1} \mu^{3/2}h^{-3/2}e^{{\frac 1 2}i\mu h^{-1}(x_2+y_2)(x_1-y_1)} \\
\sum_{n}
\int_{(2n+1)\mu h \le W } \times\\
\Upsilon_n \Bigl( \mu^{1/2}h^{-1/2}({\frac 1 2}(x_2-y_2)-\zeta) \Bigr)
\Upsilon_n \Bigl( \mu^{1/2}h^{-1/2}({\frac 1 2}(x_2-y_2)+\zeta)\Bigr)
e^{-i\mu h^{-1}\zeta (x_1-y_1)}\,d\zeta.
\label{3-10}
\end{multline}
\end{spreadlines}
Calculating the trace of this kernel we get $(2\pi)^{-1} \cdot \#\{n: (2n+1)\mu h\le W\}$ which is well known Magnetic Weyl expression. 

Note that (\ref{3-10}) (modulo exponential-quadratic factors) depends only on $|x-y|$ and one can rewrite it as 
\begin{spreadlines}{10pt}
\begin{multline}
e_W^\MW (x,y,0)=(2\pi)^{-1} \mu^{3/2}h^{-3/2}
\sum_{n}
\int_{(2n+1)\mu h \le W } \times\\
\Upsilon_n \Bigl( \mu^{1/2}h^{-1/2}({\frac 1 2}|x-y|-\zeta \Bigr)
\Upsilon_n \Bigl( \mu^{1/2}h^{-1/2}({\frac 1 2}|x-y|+\zeta)\Bigr)
\,d\zeta
\label{3-11}
\end{multline}
\end{spreadlines}
or equivalently
\begin{spreadlines}{8pt}
\begin{multline}
e_W^\MW (x,y,0)=(2\pi)^{-1} \mu^{3/2}h^{-3/2}e^{{\frac 1 2}i\mu h^{-1}(x_2+y_2)(x_4A-y_1)} \times\\
\sum_{n}
\int_{(2n+1)\mu h \le W } 
\Upsilon_n \Bigl( -\mu^{1/2}h^{-1/2}\zeta \Bigr)
\Upsilon_n \Bigl( \mu^{1/2}h^{-1/2}\zeta\Bigr)
e^{-i\mu h^{-1}\zeta |x-y|}\,d\zeta.
\label{3-12}
\end{multline}
\end{spreadlines}

\subsection{Calculations with Oscillatory Integrals}\label{sect-3-2}

It immediately follows from \cite{Ivr1} that as $|t|\le \epsilon \mu $ propagator $u$ is given by 
\begin{multline}
u(x,y,t)= \cF^*_xU\cF^\dag_y = (2\pi \mu^{-1}h)^{-2}
\iint e^{i\mu h^{-1} \bigl(S(x, \xi,\mu^{-1})-\langle x',\xi\rangle -
S(y ,\eta,\mu^{-1}) +\langle y',\xi\rangle\bigr)}\times \\[3pt]
\sum_{l,l'} \alpha _l (x,x',\xi) \alpha _{l'}^\dag (y,y',\xi) \mu^{-l-l'}h^{l+l'} 
\sum_{n,n'} \Upsilon_n (x_1'\mu^{1/2}h^{-1/2})\Upsilon_{n'} (y_1'\mu^{1/2}h^{-1/2})
u_{nn'} (x'_2,y'_2,t)\,d\xi d\eta 
\label{3-13}
\end{multline}
where $\cF$ is $\mu^{-1}h$-FIO,
\begin{equation}
\cF v (x)= (2\pi \mu^{-1}h)^{-1} 
\int e^{i\mu h^{-1}( S(x,\xi,\mu^{-1})-\langle x',\xi\rangle)}
\sum_{l,l'} \alpha _l (x,x',\xi) v (x') \, dx' d\xi,
\label{3-14}
\end{equation}
which reduces $A$ to its canonical form $A=\cF^*\cA\cF$ with
\begin{multline}
\cA=\omega (x_2,\mu^{-1}hD_2) (h^2D_1^2+\mu^2 x_1^2) - W (x_2,\mu^{-1}hD_2)+\\[2pt]
\sum_{m+k+l\ge2} a_{mkl} (x_2,\mu^{-1}hD_2) (h^2D_1^2+\mu^2 x_1^2)^m \mu^{2-2m-2k-l} h^l
\label{3-15}
\end{multline}
with $\omega = F \circ \Psi$, $W=V\circ \Psi$ with some map $\Psi:T*\bR\to \bR^2$. 

Really, decomposing $U$ (propagator of $\cA$) into sum
\begin{equation}
U(x',y',t)= \sum_{n,n'\in \bZ^+} u_{nn'} (x'_2,y'_2,t)
\Upsilon_n (x_1'\mu^{1/2}h^{-1/2})\Upsilon_{n'} (y_1'\mu^{1/2}h^{-1/2})
\label{3-16}
\end{equation}
one gets (\ref{3-13}).

On the other hand, since $\cA$ is a diagonal matrix in the basis of $\Upsilon_n (x_1'\mu^{1/2}h^{-1/2})$ I conclude that the non-diagonal elements $u_{nn'}$ are negligible on the time interval I are interested while diagonal elements are Schwartz kernels of propagators $e^{ih^{-1}\cA_n t}$ of 1-dimensional operators 
\begin{multline}
\cA_n=\omega (x_2,\mu^{-1}hD_2) \bigl((2n+1)\mu h\bigr) -
W (x_2,\mu^{-1}hD_2)+ \\[2pt]
\sum_{m+k+l\ge2} a_{mkl} (x_2,\mu^{-1}hD_2) \bigl((2n+1)\mu h\bigr) ^m \mu^{2-2m-2k-l} h^l
\label{3-17}
\end{multline}
with the principal parts
\begin{multline}
\cA^0_n=\omega (x_2,\mu^{-1}hD_2) \bigl((2n+1)\mu h\bigr) -
W (x_2,\mu^{-1}hD_2)+ \\[3pt]
\sum_{m+k \ge 2} a_{mkl} (x_2,\mu^{-1}hD_2) \bigl((2n+1)\mu h\bigr) ^m \mu^{2-2m-2k}.
\label{3-18}
\end{multline}
Rescaling $t\to t'=\mu^{-1} t$ one gets $e^{i\mu h^{-1}\cA_n t'}$ with $|t'|\le \epsilon$ which are standard $\mu^{-1}h$-FIOs; more precisely
\begin{multline}
u_{nn}(x'_2,y'_2,t)\equiv (2\pi \mu^{-1}h)^{-1}\int 
e^{i\mu h^{-1}\bigl(\phi_n ( x'_2,y'_2,\zeta, \mu^{-2})+ \cA_n (y,\zeta,\mu^{-2})t'\bigr)}\times \\
\sum _l \beta_l ( n, x'_2,y'_2,t',\zeta,\mu^{-2})\mu^{-l}h^{-l}\bigr|_{t'=\mu^{-1}t}
\,d\zeta
\label{3-19}
\end{multline}
where 
\begin{align}
&\cA^0_n (x_2,\partial_{x'_2} \phi_n )=\cA_n(y,\zeta), \label{3-20} \\
&\phi_n =0,\quad \partial_{x'_2} \phi_n=\zeta\qquad \text{as } x'_2=y'_2,\label{3-21}.
\end{align}

 So, one needs to plug (\ref{3-13}) with sum over $n=n'$ with $u_{nn}$ defined by (\ref{3-19})-(\ref{3-21}) into Tauberian expression (\ref{1-4}) for $e(x,y,0)$ and then to plug the result into calculation of $\cI$. One can prove easily that skipping terms with $l\ge 1$ or $l'\ge 1$ in (\ref{3-19}) and (\ref{3-13}) results in $O(\mu^{-1}h^{-1-\kappa})$ error in $\cI$. Also, with the same error one can replace $T=\epsilon \mu$ by $T=\infty$ after these substitutions are made leading to the following analogues of (\ref{3-13}) and (\ref{3-19}):
\begin{multline}
e' (x,y,0) = (2\pi \mu^{-1}h)^{-2}
\iint e^{i\mu h^{-1} \bigl(S(x, \xi,\mu^{-1})-\langle x',\xi\rangle -S(y ,\eta,\mu^{-1})+\langle y',\xi\rangle\bigr)}\times\\[3pt]
\alpha (x,x',\xi) \alpha (y,y',\xi)^\dag \sum_n \Upsilon_n (x'_1\mu^{1/2}h^{-1/2})\Upsilon_n (y'_1\mu^{1/2}h^{-1/2})
e_n (x'_2,y'_2,t)\,d\xi d\eta 
\label{3-22}
\end{multline}
with
\begin{equation}
e_n (x'_2,y'_2,t)\equiv (2\pi \mu^{-1}h)^{-1}\int _{\bigl\{\cA_n^0 (x'_2,y'_2,\zeta)<0\bigr\} }
e^{i\mu h^{-1} \phi_n ( x'_2,y'_2,\zeta, \mu^{-2})}
\beta ( n, x'_2,y'_2,0,\zeta,\mu^{-2})
\,d\zeta
\label{3-23}
\end{equation}
where $\alpha=\alpha_0$ and $\beta=\beta_0$. 

So, I arrive to

\begin{proposition}\label{prop-3-1} Let conditions $(\ref{1-1})-(\ref{1-3})$ be fulfilled. Then replacing in $I$ $e(x,y,0)$ by $e' (x,y,0)$ defined by $(\ref{3-22})-(\ref{3-23})$ we make an error $O(\mu^{-1}h^{-1-\kappa})$.
\end{proposition}

However even if one can make more descriptive construction defining all the phases and amplitudes in geometric terms\footnote{\label{foot-9} The reader who is more geometrically savvy than me can do this.}, I prefer to make a less sharp but more explicit calculations.

\subsection{Calculations in Outer Zone}
\label{sect-3-3}

\subsubsection{}\label{3-3-1}  Consider \emph{outer zone\/} $\{|x-y|\ge {\bar\gamma}_1\}$ where ${\bar\gamma}_1\ge C_0\mu^{-1}$ will be defined later. It follows immediately from proposition \ref{prop-2-1} that 

\begin{claim}\label{3-24}
The contribution of zone $\{|x-y|\ge {\bar\gamma}_1\}$ with $\gamma\ge C_0\mu^{-1}$ to the asymptotics does not exceed 
\begin{equation}
C\mu^{-1}h^{-1}\gamma^{-1-\kappa}
\label{3-25}
\end{equation}
\end{claim}
and therefore does not exceed $C\mu^{-1}h^{-1-\kappa}$ as 
$\gamma = {\bar\gamma}_2 \Def h^{\kappa/(\kappa +1)}$.
However I want a better estimate.

Let us consider first $\omega(x)=1$ and $\psi_j=\psi_{j,\gamma}$ which are $\gamma$-admissible
functions with
%functions of the type described in proposition  \ref{pro-3-12} with 
$\gamma$-disjoint supports and $\gamma\ge C_0\mu^{-1}$. Let us apply formula (\ref{DE1-2-13}) \cite{DE1} for $\cI$
\begin{equation}
h^{-2}\Tr\int _{(\tau_1,\tau_2)\in \bR^{-,2}}
F_{t_1\to h^{-1}\tau_1,t_2\to h^{-1}\tau_2}
\Bigl( \chi_T(t_1) \chi_T(t_2) 
\psi_{1,\gamma}\psi_{2,\gamma,t_1} U(t_1+t_2) ) \Bigr) \,d\tau
\label{3-26}
\end{equation}
with $\psi_t =U(t)\psi U(-t)$. Originally this formula is with ${\bar\chi}_{\epsilon \mu}$. However due to propagation for magnetic Schr\"odinger operator one can replace ${\bar\chi}_{\epsilon\mu}(t)$ by $\chi_T(t)$ where $T=\mu \gamma$ and $\chi (t)$ is supported in $(-2C_0,-\epsilon_0)\cup (\epsilon_0,2C_0)$ and equals 1 on $(-C_0,-2\epsilon_0)\cup (2\epsilon_0,C_0)$.

After this substitution can rewrite (\ref{3-36}) as 
\begin{align}
&-T^{-2} F_{t_1\to h^{-1}\tau_1, t_2\to h^{-1}\tau_2}
\Bigl( {\tilde\chi}_T(t_1) {\tilde\chi}_T(t_2) 
\psi_{1,\gamma}\psi_{2,\gamma,t_1} U(t_1+t_2) \Bigr) \Bigr|_{\tau_1=\tau_2=0}=\label{3-27}\\[3pt]
&-(2\pi )^{-2} T^{-2}\iint \Bigl( {\tilde\chi}_T(t_1){\tilde\chi}_T(t_2) 
\psi_{1,\gamma}\psi_{2,\gamma,t_1} U(t_1+t_2) \Bigr) e^{-ih^{-1}(t'+t'')\tau}\,dt_1dt_2\Bigr|_{\tau=0}=\notag\\[3pt]
&-(2\pi )^{-1} T^{-2}F_{t\to h^{-1}\tau} \Bigl( \int{\tilde\chi}_T({\frac 1 2}t+s){\tilde\chi}_T({\frac 1 2}t-s) 
\psi_{1,\gamma}\psi_{2,\gamma,{\frac 1 2}t+s} \, ds \Bigr)U(t)\Bigr|_{\tau=0}\notag
\end{align}
with  ${\tilde\chi}(t)=\chi (t)/t$. 

First, let us calculate $U(t)\psi _\gamma U(-t)$. To do this let us go to the canonical form of operator $A$ (see \cite{Ivr1}, Chapter 6):
\begin{equation}
\cA = Z^*Z + W(x_1,\mu^{-1}hD_1) +\sum _{m+n+l\ge 1} b_{mn}(x_1,\mu^{-1}hD_1)(Z^*Z)^m \mu^{-2m-2n-l}h^l 
\label{3-28}
\end{equation}
with 
\begin{equation}
Z= hD_2+i\mu x_2,\qquad Z^* = hD_2-i\mu x_2.
\label{3-29}
\end{equation}
Then $\psi_\gamma (x)$ is transformed into  $\psi'_\gamma(x,\mu^{-1}hD) $ which can be decomposed  as $\gamma\ge \mu^{-1+\delta}$ into an  asymptotic sum
\begin{equation}
\psi'_\gamma(x,\mu^{-1}hD)\sim\sum_{\alpha\in \bZ^{+\,2}} \psi_{\gamma,\alpha} (x_1,\mu^{-1}hD_1) (\zeta^{\alpha_1}\zeta^{\dag\,\alpha_2})^\w
\label{3-30}
\end{equation}
where $\zeta=\xi_2+i x_2$, $\zeta^\dag=\xi_2 -i x_2$ are symbols of $Z$ and $Z^*$ respectively and $^\w$ here means $\mu^{-1}h$-quantization. 
Then $(\zeta^{\alpha_1}\zeta^{\dag\,\alpha_2})^\w$ is a symmetric product of $\alpha_1$ copies of $\mu^{-1}Z$ and $\alpha_2$ copies of $\mu^{-1}Z^*$.

One can see easily that
\begin{align}
&|\nabla^\beta \psi_{\gamma,\alpha} |= O(\gamma^{-|\alpha|-|\beta|})\qquad &&\forall \alpha,\beta,\label{3-31}\\
\intertext{and moreover }
&|\nabla^\beta \bigl(\psi_{\gamma,\alpha}  -{\frac 1 {\alpha!}}i^{\alpha_1-\alpha_2} \eth^{\alpha_1}  \eth^{\dag\, \alpha_2}\psi_\gamma)\circ {\bar \Psi}| = O(\gamma^{1-|\alpha|-|\beta|})\qquad &&\forall \alpha,\beta\label{3-32}
\end{align}
where $\eth=X-iY$ and $\eth^\dag=X+iY$, $X$ and $Y$ are real vector fields and actually important is only that $\eth\eth^\dag=-{\frac 1 4}F\Delta_g$ where $\Delta_g$ is a positive Laplacian associated with the metrics $(F^{-1}g^{jk})$.

Since $[Z^*,Z]=2\mu h$ and $[\cA,Z]= 2\mu h Z$, $[\cA,Z^*]= -2\mu h Z^*$
\begin{align}
&U(t)ZU(-t)=e^{ih^{-1}tZ^*Z}Ze^{-ih^{-1}tZ^*Z}= e^{2i\mu t}Z,\label{3-33}\\
&U(t)Z^*U(-t)=e^{ih^{-1}tZ^*Z}Z^*e^{-ih^{-1}tZ^*Z}= e^{-2i\mu t}Z
\notag
\end{align}
and I conclude that dropping all terms with $\alpha_1\ne \alpha_2$ and calculating an error one gets an extra $\mu^{-2}\gamma^{-1}$ factor in the expression $T^{-1}\varrho_T(t)$ with
\begin{equation}
\varrho_T(t)=T^{-1}\int{\tilde\chi}_T({\frac 1 2}t+s)
{\tilde\chi}_T({\frac 1 2}t-s) 
\psi_{1,\gamma}\psi_{2,\gamma, {\frac 1 2}t+s} \, ds
\label{3-34}
\end{equation}
and in (\ref{3-27}) itself; the latter would lead to $O(\mu^{-1}h^{-1-\kappa})$ error in the final answer.

Further, dropping terms with $\alpha_1=\alpha_2\ge n$ would lead to the extra factor $\mu^{-2n}\gamma^{-2n}$ in the error.  On the other hand, replacing $Z^*Z$ by $F^{-1}W \circ \Psi$ brings even smaller error since I set $\tau=0$.

Under transformation of the operator to its  canonical form (\ref{3-28}), $\varrho_T$ is transformed into 
\begin{equation}
\varrho'(x,\mu^{-1}hD)\sim\sum_{\alpha\in \bZ^{+\,2}} \varrho_\alpha (x_1,\mu^{-1}hD_1) (\zeta^{\alpha_1}\zeta^{\dag\,\alpha_2})^\w
\label{3-35}
\end{equation}
where   one can prove easily using the same method as in the proof of (\ref{DE1-2-17}), (\ref{DE1-2-18}) of \cite{DE1}, that 

\begin{claim}\label{3-36}
Symbols $\varrho_\alpha$ satisfy the same inequalities (\ref{3-31}),(\ref{3-32}) as $\psi_{\gamma,\alpha}$. 
\end{claim}

Then to calculate (\ref{3-27}) one can apply the standard approach of \cite{Ivr1}, Chapter 6 resulting in the asymptotic decomposition
\begin{equation}
\sim \mu^{-1}h^{-1}\gamma^{-1-\kappa}\sum_{n,m,l} \varkappa_{nml}\mu^{-2n+l}\gamma^{-2n+m}h^l
\label{3-37}
\end{equation}
where 
\begin{equation}
\kappa_{000}= \const \int \sum_n \varphi (x)\updelta (W-2\mu h)\, dx
\label{3-38}
\end{equation}
and other coefficients have the same form\footnote{\label{foot-10} Actually they are sums of such terms with $\updelta$ replaced by $\delta^{(k)}$ which is not essential due to (\ref{1-3}); I remind that $W=V/F$.}.

This answer (\ref{3-38}) has exactly magnitude $\mu^{-1}h^{-1}\gamma^{-1-\kappa}$ as it should have; and since I am  interested in the answer modulo $O(\mu^{-1}h^{-1}\gamma^{-\kappa})$ I can skip any term with an extra factor $\gamma$ (or lesser). In particular, in the above calculations I can replace $\psi_{2,\gamma,{\frac 12}t+s}$ by $(\psi_\alpha \circ \Phi_t)^\w$ with $\Phi_t=e^{\mu^{-1}t H_W}$ (because this error brings an extra factor $\mu^{-1}h\gamma^{-1}$); moreover, I can replace map $\Phi_t$ by $\Phi_{xt}=e^{\mu^{-1}t H_W(x)}$ (because these maps coincide modulo $O(\gamma)$).

It immediately implies 

\begin{proposition}\label{prop-3-2} If in the outer zone $\{|x-y|\ge \mu^{-1+\delta}\}$ one replaces $e(x,y,0)$ by $(\ref{3-8})$ like expression for operator which in an appropriate coordinates has form $(\ref{3-4})$\,\footnote{\label{foot-11} I leave to the reader to reach such expression in the arbitrary coordinates.} the error would not exceed $O(\mu^{-1}h^{-1-\kappa})$.
\end{proposition}

\subsection{Successive Approximations }\label{sect-3-4}
 Now  I still want to cover zone $\{C_0\mu^{-1}\le |x-y|\le \mu^{-1+\delta'}\}$. For this I am going to
compare operator in question and the model operator without using canonical form.  

\subsubsection{}\label{sect-3-4-1}
I discuss successive approximation method which should be modified to the current problem. Consider first an abstract form. There is a perturbed operator $A$ and unperturbed operator $A_0$ and perturbation $B=A-A_0$ with $\|B\|=\nu$. Let us consider $U(t)=e^{ih^{-1}tA}$ and $U_0(t)=e^{ih^{-1}tA_0}$; then
\begin{align}
&U(t)=U_0(t)+ i h^{-1} \int  U_0(t_1)BU(t-t_1)\, dt_1\notag\\
\intertext{and iterating I get}
U(t)=\smashoperator{\sum_{0\le k< K} }&i^kh^{-k} \int_{\Delta_k(t)} U_0(t_1)BU_0(t_2)B\cdots U_0(t_k)BU_0(t-t_1-\dots-t_k)\,dt_1\dots dt_k+\label{3-39}\\
&i^nh^{-n}\int_{\Delta_K(t)} U_0(t_1)BU_0(t_2)B\cdots U_0(t_n)BU (t-t_1-\dots-t_K)\,dt_1\dots dt_K\notag
\end{align}
where 
\begin{equation*}
\Delta_k(t) =\bigl\{{\mathbf t}=(t_1,\dots,t_k): t_1 t^{-1}\ge 0,\dots, t_kt^{-1}\ge 0, (t_1+\dots+t_k)t^{-1}\le 1\bigr\},
\end{equation*}
term with $k=0$ is $U_0(t)$  and the last term is negligible as $|t|\le T$,
\begin{equation}
T\nu \le h^{1+\delta}
\label{3-40}
\end{equation}
and $K$ is large enough.

Let us consider term with $1\le k<n$; one can rewrite it as
\begin{equation}
i^kh^{-k} \int_{\Delta_k(t)} B_{t_1} B_{t_1+t_2}\cdots B_{t_1+\dots+t_k}\,dt_1\dots dt_k \times U_0(t)
\label{3-41}
\end{equation}
with $B_s= U(s)BU(-s)$. Rewriting 
\begin{equation}
B=\sum _{\alpha,\beta}  B_{\alpha,\beta}Z^\alpha Z^{*\,\beta},\qquad [B_{\alpha,\beta},Z]\equiv [B_{\alpha,\beta},Z^*]\equiv 0
\label{3-42}
\end{equation}
I arrive to 
\begin{equation}
B_s=\sum _{\alpha ,\beta\in \bZ^+}  B_{\alpha,\beta}(\mu^{-1}s) e^{i\mu (\alpha-\beta)s} Z^\alpha Z^{*\,\beta}
\label{3-43}
\end{equation}
and (\ref{3-41}) becomes 
\begin{multline}
i^kh^{-k} \int_{\Delta_k(t)} \sum_{\boldalpha,\boldbeta\in \bZ^{+\,k}}
\bigl(B_{\alpha_1,\beta_1}\cdots B_{\alpha_k,\beta_k}\bigr)(\mu^{-1}{\mathbf t})
e^{i\mu f({\mathbf t})}
\,dt_1\dots dt_k \times\\ Z^{\alpha_1}Z^{*\,\beta_1}\cdots Z^{\alpha_k}Z^{*\,\beta_k} U_0(t)
\label{3-44}
\end{multline}
with $\boldalpha =(\alpha_1,\dots,\alpha_k)$, 
$\boldbeta = (\beta_1,\dots,\beta_k)$ and
$f({\mathbf t})=\sum _{1\le j\le k} t_j (\sum _{j\le l\le k}(\alpha_l-\beta_l)$. Taking Taylor decomposition of $\bigl(B_{\alpha_1,\beta_1}\cdots B_{\alpha_k,\beta_k}\bigr)(\mu^{-1}{\mathbf t})$ and calculating this integral I get 
\begin{multline}
i^kh^{-k} \Bigl(\updelta_{\boldalpha\boldbeta}
B_{\alpha_1,\beta_1}(0)\dots B_{\alpha_k1,\beta_k}(0){\frac {t^k}{k!}}+ 
\sum _{\sigma \in {\mathfrak N}(\boldalpha-\boldbeta), l<k, s\ge0}
R_{k\boldalpha\boldbeta\sigma ls}\mu ^{-k+l-s}t^l e^{i\mu \sigma t}\Bigr)\times\\
Z^{\alpha_1}Z^{*\,\beta_1}\cdots Z^{\alpha_k}Z^{*\,\beta_k} U_0(t)
\label{3-45}
\end{multline}
where ${\mathfrak N}(\boldalpha-\boldbeta)$ is a finite subset of $\bZ^k$.

Plugging into (\ref{1-4}) I get a term in successive approximations of $e_T$
\begin{multline}
i^kh^{-k} T^{-1} \int {\tilde\chi}_T(t)\Bigl(\updelta_{\boldalpha\boldbeta}
B_{\alpha_1,\beta_1}(0)\dots B_{\alpha_k1,\beta_k}(0){\frac {t^k}{k!}} +\\
\sum _{\sigma \in {\mathfrak N}(\boldalpha-\boldbeta)\subset \bZ^k, l<k, s\ge0}
R_{k\boldalpha\boldbeta\sigma ls}\mu ^{-k+l-s}t^l e^{i\mu \sigma t}\Bigr)\times 
Z^{\alpha_1}Z^{*\,\beta_1}\cdots Z^{\alpha_k}Z^{*\,\beta_k} U_0(t)\,dt 
\label{3-46}
\end{multline}
which is equal to
\begin{multline}
i^kh^{-k} T^{-1}\iint {\tilde\chi}_T(t)\Bigl(\updelta_{\boldalpha\boldbeta}
B_{\alpha_1,\beta_1}(0)\dots B_{\alpha_k1,\beta_k}(0){\frac {t^k}{k!}} +\\
\sum _{\sigma \in {\mathfrak N}(\boldalpha-\boldbeta)\subset \bZ^k, l<k, s\ge0}
R_{k\boldalpha\boldbeta\sigma ls}\mu ^{-k+l-s}t^l e^{i\mu \sigma t}\Bigr)\times 
Z^{\alpha_1}Z^{*\,\beta_1}\cdots Z^{\alpha_k}Z^{*\,\beta_k} e^{i\tau h^{-1}t}\,dt d_\tau E_0(\tau).
\label{3-47}
\end{multline}
Note that
\begin{equation*}
\iint {\tilde\chi}_T(t)t^l e^{i(\mu \sigma +h^{-1}\tau)t}\,dt \,d_\tau  E_0(\tau)=
T^{l+1} \int \widehat {{\tilde\chi}_l} (\tau T h^{-1})  \,d_\tau E_0(\tau -\sigma \mu h )
\end{equation*}
with ${\tilde\chi}_l(t)=t^l{\tilde\chi}(t)= it^{l-1}\chi (t)$.

Therefore the trace norm of the first term in (\ref{3-47}) does not exceed
\begin{equation}
C h^{-1-k}T^{k-1}\prod_{1\le j\le k} \|B_{\alpha_j\beta_j}\|;
\label{3-48}
\end{equation}
where I used an already mentioned inequality  $\|E(\tau)-E(\tau')\|_1\le Ch^{-2}(hT^{-1}+|\tau-\tau'|)$.

On the other hand, as $\mu \le \epsilon h^{-1}$ the trace norm of the second term in (\ref{3-47}) does not exceed
\begin{equation}
C \mu^{-k+l} h^{-1-k}T^{l-1}\prod_{1\le j\le k} \|B_{\alpha_j\beta_j}\|.
\label{3-49}
\end{equation}
Then the contribution of zone $\{|t|\asymp T\}$ and the mentioned terms  to the final  answer  does not exceed (\ref{3-48}) and (\ref{3-49}) respectively with an extra factor which is equal to $h^{-\kappa}$ as $T\le C_0$ and to $\mu^\kappa T^{-\kappa}$ as $T\ge C_0$. More precisely, contribution of zone $\{|x-y|\ge C_0h\}$ is trivial and $\{|x-y|\le C_0h\}$ follows arguments of of the proof of proposition \ref{prop-1-2}.

Note that in my settings $\nu=\mu^{-m}$ (with $m=2$ if $F=1$ and I compare with operator considered in subsection \ref{sect-3-1}) and that $\|B_{\alpha_j\beta_j}\|\le C\mu^{-\max(m,|\alpha_j|+\beta_j|)}$.

Then the first terms in (\ref{3-47}) result in the final contribution to the error not exceeding
\begin{equation}
C h^{-1-k-\kappa}T^{k-1}\mu^{-mk};
\label{3-50}
\end{equation}
further summation with respect to $T\le C_0$ results in 
\begin{equation}
Ch^{-1-k-\kappa}\mu^{-mk}\bigl(1+\updelta_{k1}|\log h|\bigr).
\label{3-51}
\end{equation}

On the other hand, as $T\ge C_0\max (1,\mu h|\log h|)$ one should replace expression (\ref{3-50}) by 
\begin{equation}
C h^{-1-k}T^{k-1}\mu^{-mk}\times (\mu^{-1} T)^{-\kappa}\asymp
C h^{-1-k}T^{k-1-\kappa}\mu^{-mk+\kappa};
\label{3-52}
\end{equation}
then summation with respect to $T$ (from $T=C_0$ to $T=\mu ^\delta$) as $k=1$ results in $Ch^{-1-k}\mu^{-k+\kappa}\le C\mu^{-m}h^{-1-\kappa}$ and all other terms are smaller.

Further, if one considers $|\alpha_j|+|\beta_j|\ge m+1$ at least in one of the factors then estimate acquires an extra factor $\mu^{-1}$; so the final estimate would be
\begin{equation}
Ch^{-1-\kappa}\mu^{-mk-1}\bigl(1+\updelta_{k1}|\log h|\bigr).
\label{3-53}
\end{equation}

On the other hand one gets from the other terms with $l\ge 1$ the same estimate (\ref{3-51}) and with terms with $l=0$ one gets the value as $T$ reaches its lowest value $\epsilon \mu^{-1}$ i.e. one get (\ref{3-53}) as $k\ge 2$ and (\ref{3-51}) as  $k=1$ (but without logarithm term). 

There is one special case of $\epsilon (h|\log h|)^{-1}\le \mu \le h^{-1}$ when one needs to sum (\ref{3-50}) to the upper bound $C\mu h|\log h|$ rather than to $C$ but it does not affect the term with $k=1$ and all other terms are smaller.

So, as $F=1$ error estimate (\ref{3-51}) with $k=1$ is achieved:
\begin{equation}
C\mu^{-m}h^{-2-\kappa}|\log h|.
\label{3-54}
\end{equation}

Probably one can get rid off logarithmic factor.

\subsubsection{}\label{sect-3-4-2}
Let us consider $F$ which is not identically 1. Without any loss of the generality one can assume that $F(y)=1$. Then let us consider perturbation 
$B= \beta hD_t$, with $\beta=(F^{-1}-1)$, bringing $F$ to that case. Note that if $T\ge C_0$ then 
\begin{equation*}
\|B\|= O(\mu^{-1}T\times h |\log h| T^{-1})= O(\mu^{-1}h|\log h|)
\end{equation*}
since $\|hD_t\|$ on the interval in question could be brought to $Ch|\log h|$ due to logarithmic uncertainty principle. Then 
$\|B\| T \le h^\delta$.

Let consider the whole interval $(-T,T)$ and plug successive approximations to Tauberian formula (\ref{1-4}) directly. Then $hD_t$ could be dropped on ${\bar\chi}_T$ and I get 
\begin{equation}
CT^{-1}\int _\infty^0 (F_{t\to h^{-1}\tau}\chi_T(t) \int _0^t U_0(t-t')\beta U_0(t')\,dt'd\tau
\label{3-55}
\end{equation}
in the first successive perturbation  term and the other look similarly. Then using approach of the previous subsection one can  estimate $k$-th term by
\begin{equation}
CT^{-1} h^{-1} \times (\mu^{-1}T)^{k-\kappa}  \le C\mu^{-1}h^{-1-\kappa}
\label{3-56}
\end{equation}
and one should not worry about it.

\subsubsection{}\label{sect-3-4-3} So far I assumed that $\mu \le \epsilon h^{-1}$. As $\epsilon \le \mu h \le 1$ I can reduce to canonical form of the model equation and consider non-diagonal terms which are $O(\mu^{-1})$ as perturbations which can be excluded, since ellipticity locally is broken for no more than 1 number ``$n$'', leading to the diagonal perturbations $O(\mu^{-2}/\mu h)=O(\mu^{-3}h^{-1})$ which in turn leads to $O(h^{-2-\kappa} \times \mu^{-3}h^{-1})=O(\mu^{-3}h^{-3-\kappa})$ error in the final answer; but in this case it is $O(\mu^{-1}h^{-1-\kappa})$.

\subsubsection{}\label{sect-3-4-4} So, let us summarize.

\begin{proposition}\label{prop-3-3} {\rm (i)} One can replace in $(\ref{1-4})$ $U(t)$ by $U_0(t)$ with the final error not exceeding $C\mu^{-m}h^{-2-\kappa-\delta}$. In particular, as  $\mu \ge h^{-1/(m-1)-\delta}$ this error is $O(\mu^{-1}h^{-1-\kappa})$ and as 
$\mu\ge h^{-1/m-\delta}$ this error is $O(h^{-1-\kappa})$.

\medskip
{\rm(ii)} On the other hand, adding the second term of successive approximations one gets a final error $C\mu^{-2m}h^{-3-\kappa-\delta}$. In particular, as $\mu \ge h^{-2/(2m-1)-\delta}$ this error is $O(\mu^{-1}h^{-1-\kappa})$.
\end{proposition}

\section{Calculations: Intermediate Magnetic Field}
\label{sect-4}

Now I am going to combine methods of two previous sections. Namely, while in section \ref{sect-3} I did not consider points as distinguishable unless $|t|\ge C_0$, I will do it here using arguments of section \ref{sect-2}.  As a result in addition to threshold $T^*=\min (\nu ^{-1}h^{1+\delta},C_0)$  appearing in section \ref{sect-3} another thresholds $T^\#=T^\#_+(\mu,h,\phi)$ and $T=T_-(\mu,h,\phi)$ appear as $J^\pm \times J^\pm$ and $J^\pm\times J^\mp$ pairs are considered and for $|t|\ge T^\#$ distinguishability arguments are used. This is helpful as $T^*\ge T^\#_\varsigma (\phi)$ at least for some $\varsigma$ and $\phi$.

\subsection{Preliminary Analysis}\label{sect-4-1}

The previous section construction works as I  am able to run successive approximations with $T= C_0h^{-\delta}$ i.e. as $\nu \le h^{1+\delta}$. As $\nu =\mu^{-m}$ this means $\mu \ge h^{-1/m-{\bar\delta}}$. Now I am going to investigate both cases
\begin{phantomequation}
\label{4-1}
\end{phantomequation}
\begin{equation}
\mu \ge h^{-1/m-{\bar\delta}},\qquad \mu \le h^{-1/m-{\bar\delta}}.
\label{4-1-*}
\tag*{$(\ref*{4-1})_{1,2}$}
\end{equation}

As $T^*\ge h^{-\delta}$ is fulfilled I estimated the contribution of $\{|t|\ge T^*\}$ by $C\mu^{-1}h^{-1-\kappa}$ in subsection \ref{sect-3-3} already.
So, as $T^*\ge h^{-\delta}$ I can reset in what follows $T^*$ to $h^{-\delta}$.

Further, as $T^*=h^{-\delta}$, $\mu \le \epsilon (h|\log h|)^{-1}$ and $C_0\le T\le T^*$ one can estimate the contribution of $|t|\asymp T$ by (\ref{3-52})
which sums with respect to $T$ to its value as $T=C_0$ i.e. to $C\mu^{\kappa-mk}  h^{-1-k }$. So,

So, 
\begin{claim}\label{4-2}
Under condition $(\ref{4-1})_1$ one can redefine $T^*=C_0$; then contribution of zone $\{|t|\ge T^*\}$ to the remainder estimate does not exceed 
\begin{equation}
C\mu^{\kappa-mk}  h^{-1-k }+C\mu^{-1}h^{-1-\kappa}.
\label{4-3}
\end{equation}
\end{claim}

I will use more sophisticated approach (weak magnetic field approach) as $|t|\le C_0$. Note that  all our analysis makes sense as $T^*\ge \epsilon \mu^{-1}$ only i.e. 
\begin{equation}
\mu \ge h^{-1/(m+1)-{\bar\delta}}
\label{4-4}
\end{equation}

\subsection{Analysis of $J^\pm \times J^\pm$ Pairs}
\label{sect-4-2}

\subsubsection{}\label{4-2-1}
At this  subsection I consider only $J^\pm \times J^\pm$ pairs without polar caps; $J^\pm \times J^\mp$ pairs and polar caps I consider later. Let us pick up $T=T^*$. Plugging  $\phi \asymp T$ under condition $(\ref{4-1})_2$,  one can see that distinguishability condition $\phi T  \ge (\mu h |\log h|)^{2/3}$ is fulfilled  as
\begin{equation}
T\ge T^*_2=C(\mu h |\log h|)^{1/3}.
\label{4-5}
\end{equation}
Note that (\ref{4-5}) always follows from (\ref{4-4}) provided $m\ge 3$.
Further as $m=2$ (\ref{4-5})  follows from (\ref{4-4}) as $\mu\ge h^{-2/5-\delta}$. 

On the other hand, as $m\ge 4$ unperturbed operator is too complicated to handle and I am interested in the case $m\le 3$ only.

\begin{proposition}\label{prop-4-1} Let  condition  $(\ref{4-4})$ be fulfilled and let either $m\ge 3$ or $m=2$, $\mu \ge h^{-2/5-\delta}$.

Then contribution to the error of zone $|t|\asymp T\le C$  and $J^\pm \times J^\pm$ pairs to the error does not exceed 
\begin{align}
&Ch^{-1}\mu^\kappa\bigl(T^{-1-\kappa}+T^{-2\kappa}+T^{-2}\updelta_{\kappa1}|\log h|\bigr)\qquad &&\text{as\ }T\ge T^*\label{4-6}\\
&Ch^{-1-k}\mu^{\kappa-mk} \bigl(T^{k-1-\kappa}+T^{k-2\kappa}+\updelta_{\kappa1}T^{k-2} |\log T|\bigr)
\qquad&&\text{as\ } T_2^*\le T\le T^*\label{4-7}\\
&Ch^{-1-k}\mu^{-mk} \bigl(\mu^ \kappa T^{k-1-\kappa}+ T^{k-2} h^{-\kappa}(\mu h|\log h|)^{2/3}\bigr)\qquad&&\text{as\ } T_3^*\le T\le T^*_2\label{4-8}\\
&Ch^{-1-k-\kappa}\mu^{-mk}  T^{k-1} \qquad&&\text{as\ } {\bar T}=\mu^{-1}\le T\le T^*_3\label{4-9}
\end{align}
where 
\begin{equation}
T_3^*= \left\{
\begin{aligned}
&C(\mu h|\log h|)^{2/3}\qquad &&\text{as\ } \mu \ge \epsilon (h|\log h|)^{-2/5},\\
&\epsilon \mu^{-1}  &&\text{as\ }\mu \le \epsilon (h|\log h|)^{-2/5};
\end{aligned}\right.
\label{4-10}
\end{equation}
in the latter case zone $(\ref{4-9})$ disappears.
\end{proposition}

\begin{proof}
(i) Note that  as $T\ge T^*_2$ all $J^\pm \times J^\pm$ pairs are distinguishable and one can apply estimate (\ref{2-17}): an error does not exceed
\begin{equation}
CT^{-1}h^{-1}\int _T^1 (\mu ^{-1}\phi T)^{-\kappa}\,d\phi;
\label{4-11}
\end{equation}
plugging $T=T^*$ one gets (\ref{4-6}). 

\medskip\noindent
(ii) For $T\le T^*$ one can apply successive approximation method and (\ref{4-6}) acquires factor $\mu^{-mk}h^{-k}T^k$ thus resulting in (\ref{4-7}).

\medskip\noindent
(iii)  For $T^*_3\le T\le T^*_2$ there is no distitinguishability  on the ends $\{\phi \le \phi_T= C(\mu h|\log h)^{2/3}T^{-1}\}$  and integral in (\ref{4-10}) is taken from $\phi_T$ to $1$, while contribution of the ends is estimated by $Ch^{-1-k-\kappa}T^{k-1}\phi_T$ thus resulting in (\ref{4-8}). 

\medskip\noindent
(iv)  Finally as 
$T\le T^*_3$ there is no distinguishability even in the center and $\phi_T\asymp1$.
\end{proof}

\begin{remark}\label{rem-4-2} Note that for $m=2,3$ I assumed $\mu\ge h^{-1/(m+1)-\delta}$ and then case (\ref{4-9}) may occur.
\end{remark}

For $m=2$, $h^{-1/3-\delta}\le \mu \le h^{-2/5-\delta}$  the previous order $T_3^*\ll T^*_2\ll T^*$ is replaced by $T_3^*\ll T^*\ll T^*_2$ and the same arguments imply

\begin{proposition}\label{prop-4-3} Let    $m=2$, $h^{-1/3-\delta}\le \mu \le h^{-2/5-\delta}$.

Then contribution to the error of zone $|t|\asymp T\le C$  and $J^\pm \times J^\pm$ pairs to the error does not exceed $(\ref{4-6})$ as $T\ge T^*_2$,
\begin{align}
&Ch^{-1}
\bigl(\mu^ \kappa T^{-1-\kappa}+ T^{-2} h^{-\kappa}(\mu h|\log h|)^{2/3}\bigr)
\qquad&&\text{as\ } T^*\le T\le T^*_2\label{4-12},
\end{align}
$(\ref{4-8})$ as $T^*_3\le T\le T^*$ and $(\ref{4-9})$ as ${\bar T}\le T\le T^*_3$; as $\mu \le \epsilon (h|\log h|)^{-2/5}$ zone $(\ref{4-9})$ disappears.
\end{proposition}

\subsubsection{}\label{sect-4-2-2}  Now I need to sum estimates (\ref{4-6}), (\ref{4-7}) or (\ref{4-11}), (\ref{4-8}) and may be (\ref{4-9}) with respect to $T$ in the corresponding intervals and find their sum. Note that  (\ref{4-9}) almost always sums to its value as $T$ hits its largest value $T=T^*_3$ which is exactly the second term in (\ref{4-8}) as $T$ hits its lowest value $T=T^*_3$;  the singular exception is $k=1$ when the summation of (\ref{4-9}) results in extra term
\begin{equation}
C_1\mu^{-m}h^{-2-\kappa} \bigl(\log (\mu^{5/2}h |\log h|)\bigr)_+.
\label{4-13}
\end{equation}
In the analysis below I assume that $\kappa\ne 1$; otherwise some terms may acquire extra logarithmic factors; details are left to the reader.

\medskip\noindent
(i) Let either $k=1$ or $k=2,\kappa>1$. Then almost every term in (\ref{4-6}), (\ref{4-7}) or (\ref{4-11}), (\ref{4-8})  is maximized by $T$ hitting its lowest value; the sole exception is the second term in (\ref{4-7}) in the case $k=1$, $\kappa \le 1/2$ but then it is dominated by the first one. Then  summation results  in (\ref{4-8}) calculated as $T=T^*_3$ i.e. 
\begin{equation}
\left\{\begin{aligned}
&Ch^{-1-k-\kappa}\mu^{-mk}  (\mu h|\log h|)^{2(k-1)/3} 
\qquad &&\text{as\ }\mu \ge \epsilon (h|\log h|)^{-2/5};\\
&Ch^{-1-k}\mu^{1-k-mk} \bigl( \mu^{2\kappa}+ \mu  h^{-\kappa}(\mu h|\log h|)^{2/3}\bigr) \qquad &&\text{as\ }\mu \le \epsilon(h|\log h|)^{-2/5}.
\end{aligned}\right.
\label{4-14}
\end{equation}

\medskip\noindent
(ii) Let $k>\max(2,1+\kappa,2\kappa)$, $\kappa<1$ (equalities would bring some logarithmic factors). Then  every term in (\ref{4-7})--(\ref{4-9})  is maximized by $T$ hitting its highest value; then as $T^*\lesssim 1$ the main contribution is delivered by the first term in (\ref{4-6}) as $T=T^*$ or by the second term in (\ref{4-8}) as $T=T^*_2$ and I get $C\mu^{-m-m\kappa+\kappa}h^{-2-\kappa-\delta}$ and $Ch^{-1-k-\kappa}\mu^{-mk}(\mu h|\log h|)^{2/3}$; in the opposite case as $\mu \ge h^{-1/m-\delta'}$ the contribution of zone $\{|t|\ge T^*\}$ is given by (\ref{4-3}) and again it is larger than contribution of other zones with the possible exception of (\ref{4-8}); so the final answer is 
\begin{equation}
\left\{\begin{aligned}
&C\mu^{-m-m\kappa+\kappa}h^{-2-\kappa-\delta} +Ch^{-1-k-\kappa}\mu^{-mk}(\mu h|\log h|)^{2/3}
\qquad &&\text{as\ }\mu \le \epsilon h^{-1/m-\delta'};\\
&C\mu^{\kappa-mk} h^{-1-k} + C\mu^{-1}h^{-1-\kappa} +Ch^{-1-k-\kappa}\mu^{-mk}(\mu h|\log h|)^{2/3}\quad &&\text{as\ }\mu \ge \epsilon h^{-1/m-\delta'}.
\end{aligned}\right.
\label{4-15}
\end{equation}

\medskip\noindent
(iii) Let $k=3, \kappa>3/2$. In this case in (\ref{4-7}) the second term is dominant but as $T=T^*_2$ it is still less than the second term in (\ref{4-8}); so the final answer is given by (\ref{4-15}) again.

So, I arrive to

\begin{proposition}\label{prop-4-4}
In frames of proposition \ref{prop-4-1} the contribution of $J^\pm\times J^\mp$ pairs to an error  does not exceed $(\ref{4-14})$ in the case (i) or $(\ref{4-15})$ in the cases (ii)-(iii).
\end{proposition}

In frames of proposition \ref{prop-4-3} the main contribution is delivered either by (\ref{4-12}) or by (\ref{4-8}). Contribution of (\ref{4-12}) always is its value as $T=T^*=\mu^2h^{1+\delta}$, namely
\begin{equation}
C\mu^{-2-\kappa }h^{-2-\kappa-\delta}+C\mu^{-4/3}h^{-7/3-\kappa-\delta}.
\label{4-16}
\end{equation}
Note that value of (\ref{4-8}) as $T=T^*$ is less than this, and its value as $T=T^*_3$ is 
\begin{equation}
C\mu^{\kappa -2k}h^{-1-k} (\mu h|\log h|)^{2(k-1-\kappa)/3}+ 
C\mu^{-2k} h^{-1-k-\kappa}(\mu h|\log h|)^{2(k-1)/3}.
\label{4-17}
\end{equation}

So, I arrive to\

\begin{proposition}\label{prop-4-5}
In frames of proposition \ref{prop-4-3} the contribution of $J^\pm\times J^\mp$ pairs to an error  does not exceed $(\ref{4-16})+(\ref{4-17})$.
\end{proposition}

\subsection{Analysis of $J^\pm \times J^\mp$ Pairs and Polar Caps}
\label{sect-4-3}

\subsubsection{}\label{sect-4-3-1} Consider now $J^\pm \times J^\mp$ pairs. Polar caps we consider later. As $T^*\ge C_0$ I already estimated contribution of $\{|t|\ge T^*\}$ to the error. In this case contribution of $T\le T^*$ to the error does not exceed $C\mu ^{-mk}h^{-1-k-\kappa}$ and no improvements is possible. Really, as analysis of subsubsection \ref{sect-2-2-3} as $T^*_4\ge C$ 
with
\begin{equation}
T^*_4 = C\mu^3h|\log h|
\label{4-18}
\end{equation}
weak magnetic field approach does not work at all. However $T^*=\mu^m h^{1+\delta}\le T^*_4$ as $m\le 3$.

This estimate is larger than I got for $J^\mp \times J^\pm$ pairs previously and so 
\begin{claim}\label{4-19}
As $\mu \ge h^{-1/m-\delta}$ the (prefinal) answer is 
$C\mu ^{-mk}h^{-1-k-\kappa}$.
\end{claim}

On the other hand as $T^*\le C_0$  one can use results of subsubsection \ref{sect-2-2-3} for $T^*\le T\le C_0$ and one gets  $Ch^{-1-\kappa}|\log h|$ as $\mu\ge \epsilon (h|\log h|)^{-1/3}$. And since weak magnetic field approach here really works only as $T\ge T^*_4$, the contribution of interval 
$[T^*,T^*_4]$ to the error is estimated by $Ch^{-1-\kappa}|\log T^*_4/T^*|\asymp C\mu^{-1-\kappa}|\log h|$ and this estimate cannot be improved by our methods.

Meanwhile  contribution of $J^\pm\times J^\mp$ pairs and $T\le T^*$ to the error is smaller than $Ch^{-1-\kappa}|\log h|$ as $k\ge 1$.  

So I arrive to the same error estimate as before:

\begin{claim}\label{4-20}
As $ h^{-1/(m+1)-\delta}\le \mu \le h^{-1/m-\delta}$ the contribution of $J^\pm\times J^\mp$ pairs to the error does not exceed 
$Ch^{-1-\kappa}|\log h|$.
\end{claim}
I remind  that the case $\mu \le (h|\log h|)^{-1/3}$ should be addressed only as $m\ge3$.

\subsubsection{}\label{sect-4-3-2} Let us consider polar caps $\{ |\sin \phi|\le {\bar\rho}=C(\mu h|\log h|)^{1/2}\}$. As $|t|\asymp T$  their contribution to an error does not exceed
\begin{multline}
Ch^{-1-\kappa}T^{-1} \times (\mu^{-m}h^{-1}T)^k \times {\bar\rho}\asymp\\
 Ch^{-1-\kappa} \mu^{-m+1/2}h^{-1/2}|\log h|^{-1/2} (\mu^{-m}h^{-1}T)^{k-1}.
\label{4-21}
\end{multline}
Here one needs to consider only $T\le {\bar\rho}$. In this case 
\begin{equation*}
\mu^{-m}h^{-1}T\le \mu^{-m}h^{-1}(\mu h |\log h|)^{1/2}\asymp (\mu^{1-2m}h^{-1}|\log h|)^{1/2}\le h^\delta
\end{equation*}
as $\mu \ge h^{-1/(2m+1)-\delta}$. 

Summation of (\ref{4-21}) with respect to $T$ acquires an extra logarithmic factor as $k=1$ but $\mu^{-m+1/2}h^{-1/2}\le \ll h^{\delta}$. So 

\begin{claim}\label{4-22}
As $ h^{-1/(m+1)-\delta} \le \mu $ the contribution of polar caps to the error does not exceed  $Ch^{-1-\kappa}|\log h|$ as well.
\end{claim}

\subsubsection{}\label{sect-4-3-3} 
So, as $m=2,3$ and $h^{-1/(m+1)+\delta}\le \mu \le h^{-1/m-\delta}$ the contribution of $J^\pm\times J^\mp$ pairs and polar caps to an error does not exceed $Ch^{-1-\kappa}$. This result is valid for $k= 1$ and is not improved as $k\ge2$. I need to add this the contribution of $J^\pm\times J^\pm$ pairs arriving to

\begin{proposition}\label{prop-4-6} Let $m=2,3$  and $h^{-1/(m+1)+\delta}\le \mu \le h^{-1/m-\delta}$. Then 

\medskip\noindent
{\rm (a)} In frames of proposition \ref{prop-4-1}  an error  does not exceed $(\ref{4-14})+Ch^{-1-\kappa}$ in the case (i) or $(\ref{4-15})+Ch^{-1-\kappa}$ in the cases (ii)-(iii) of subsubsection \ref{sect-4-2-2};

\medskip\noindent
{\rm (b)} In frames of proposition \ref{prop-4-3}  an error  does not exceed $(\ref{4-16})+(\ref{4-17})+Ch^{-1-\kappa}$.
\end{proposition}

I remind that as $\mu \ge h^{-1/m-\delta}$ an error does not exceed $C\mu^{-mk}h^{-1-k-\kappa }$. 

\subsection{More Calculations}\label{sect-4-4} 

At this moment all calculations are done in the case $m=2$, $k=0$ (so actually estimate is achieved with $k=1$; see subsection \ref{sect-3-1}). I am going to consider $m=2,k\ge1$ leaving to the reader to consider case $m=3$, $k=0$: in this case unperturbed operator is quadratic.

 So, in order to get above estimates one actually needs to calculate the unperturbed expression and the next term in the successive approximations. 

To exploit perturbation, let us note that perturbation to be considered is
\begin{equation}
B=\sum_{|\alpha|=m}\beta_\alpha (x-y)^\alpha
\label{4-23}
\end{equation}
and $m=3$ in our case because an extra factor $(x-y)$ results in an extra factor $\mu^{-1}$.  

The $k$-th perturbation term in the decomposition of $E(0)$ is equal to
\begin{multline}
-h^{-1-k}\int_{-\infty}^0 F_{t\to h^{-1}\tau}\Bigl({\bar\chi}_T(t)\times \\ \int_{\Delta_t        } U_0(t-t_1)BU_0(t_1-t_2)BU_0(t_2-t_3)\cdots U_0(t_k)\,dt_1\cdots dt_k\Bigr)\,d\tau=\\
-h^{-1-k}\int_{-\infty}^0 F_{t\to h^{-1}\tau}\Bigl({\bar\chi}_T(t)U_0(t)\int _{\Delta_t} B_{t_1}B_{t_2}\cdots B_{t_k} \,dt_1\cdots dt_k\Bigr)\,d\tau
\label{4-24}
\end{multline}
where  $\Delta_t=\{(t_1,\dots,t_k), 0\le \pm t_k\le \pm t_{k-1}\le \dots \le \pm t_1\le \pm t\}$  depending on the sign of $t$ and $B_{t'}=U_0(-t')BU_0(t')$.

To analyze (\ref{4-24}) let us consider $U_0(t)(x_j-y_j)U_0(-t)$. Note that
\begin{equation}
x_j= x'_j +\varkappa_j \mu^{-1}Z +\varkappa_j^\dag \mu^{-1} Z^* 
\label{4-25}
\end{equation}
with slow evolution of the components
\begin{equation}
U_0(t)x'_jU_0(-t)= x'_j +\mu^{-1}v_j t +O(\mu^{-2}),
\label{4-26}
\end{equation}
and $U_0(t)ZU_0(-t)= e^{2i\mu t}Z$, $U_0(t)Z^*U_0(-t)= e^{-2i\mu t}Z$ (see (\ref{3-33})).

Let us decompose 
\begin{equation*}
(x-y)^\alpha = \sum _{|\beta|+|\sigma|=3} \mu^{-|\beta |}
\rho_{\alpha\beta\sigma} (x'-y')^\sigma
(Z_x-Z_y)^{\beta _1}(Z_x-Z_y)^{*\,\beta_2}
\end{equation*}
where  $\alpha,\beta ,\sigma\in \bZ^{+\,2}$.

As $(x'-y')$ applies to $\updelta (x-y)$ one gets $0$; therefore modulo term which brings an extra factor $\mu^{-1}$ and leads to a final error $O(\mu^{-1}h^{-1-\kappa})$ even as $k=1$ 
\begin{equation*}
U_0(t)(x-y)^\alpha U_0(-t)\equiv 
\mu^{-3}\sum_{|\beta|+|\sigma|=3} \rho _{\alpha\beta\sigma}  \bigl(e^{2i\mu  t}Z_x -Z_y\bigr)^{\beta_1} \bigl(e^{-2i\mu  t}Z^*_x -Z^*_y\bigr)^{\beta_2} (v t)^\sigma .
\end{equation*}
Opening parenthesis and integrating with respect to $t$ one can skip with the  the same final error $O(\mu^{-1}h^{-1-\kappa})$ all the terms containing $Z_x$ or $Z^*_x$ with the exception of the case $\beta_1=\beta_2=1$; so only the only important term  is
\begin{multline*}
U_0(t)(x-y)^\alpha U_0(-t)\equiv 
\mu^{-3}\sum_{|\beta|+|\sigma|=3, \beta\ne (1,1)} (-1)^{|\beta|} \rho _{\alpha\beta\sigma}  Z_y^{\beta_1} Z _y^{*\,\beta_2} (v t)^\sigma + \\\mu^{-3}\sum_{|\sigma|=1}\rho_{\alpha\sigma}
\bigl(Z_xZ^*_x + Z_yZ^*_y\bigr)(v t)^\sigma.
\end{multline*}
One can see easily that  the order of $Z_y$, $Z^*_y$ in the product does not matter. Then
\begin{equation*}
B_{t_1}\cdots B_{t_k}\equiv \mu^{-3k}\sum _{|\beta|+|\sigma|+2n=3k,\, |\nu|=|\sigma|} \rho _{k, \beta,\sigma,\nu,n} Z_y^{\beta_1}Z_y^{\beta_2} (Z^*_xZ^*_x)^n v^\sigma {\mathbf t}^\nu 
\end{equation*}
with ${\mathbf t}=(t_1,\dots,t_k)$ and $\nu\in \bZ^{+\,k}$. Then integration over $\Delta_t$ results in 
\begin{equation*}
B_{t_1}\cdots B_{t_k}\equiv \mu^{-3k}\sum _{|\beta|+2n\le 3k}  \rho _{k, \beta,n} Z_y^{\beta_1}Z_y^{\beta_2} (Z^*_xZ^*_x)^n  t^{2k-|\beta|-2n}
\end{equation*}
and then (\ref{4-24}) is equivalent to
\begin{equation*}
\sum _{|\beta|+2n\le 3k}  \mu^{-3k}h^{k- |\beta|-2n}\rho _{k, \beta,\sigma} 
 Z_y^{\beta_1}Z_y^{\beta_2} (Z^*_xZ^*_x)^n  \partial_{\tau}^ {2k-|\beta|-2n} E_0(x,y,\tau)\bigr|_{\tau=0}
\end{equation*}
and the total perturbation term in $E(x,y,0)$ is
\begin{equation}
\sum_{1\le j\le k-1} \sum _{|\beta|+2n\le 3j}  \mu^{-3j}h^{j- |\beta|-2n}
\rho _{j, \beta,\sigma} 
 Z_y^{\beta_1}Z_y^{\beta_2} (Z^*_xZ^*_x)^n \partial_{\tau}^ {2j-|\beta|-2n} E_0(x,y,\tau)\bigr|_{\tau=0}
 \label{4-27}
\end{equation}
and total perturbation term in the final answer is
\begin{multline}
\sum_{1\le j\le k-1} \sum _{|\beta|+2n\le 3j}  \mu^{-3j}h^{j- |\beta|-2n}\times \\
\iint \rho _{j, \beta,\sigma} \omega (x,y)
 Z_y^{\beta_1}Z_y^{\beta_2} (Z^*_xZ^*_x)^n \partial_{\tau}^ {2j-|\beta|-2n} E_0(x,y,\tau)\bigr|_{\tau=0} \times E_0(y,x.0)\, dxdy.
 \label{4-28}
\end{multline}

\input DE2.bbl

\end{document}

%% file: DE2.bbl
\bibliographystyle{alpha}

\providecommand{\bysame}{\leavevmode\hbox to3em{\hrulefill}\thinspace}

\vglue .06truein

\begin{tabular}{rrl}
&{\hskip 220 pt} &Department of Mathematics,\cr
&&University of Toronto,\cr
&&40, St.George Str.,\cr
&&Toronto, Ontario M5S 2E4\cr
&&Canada\cr
&&ivrii@math.toronto.edu\cr
&&Fax: (416)978-4107\cr
\end{tabular}